\documentclass[a4paper,11pt]{article}
\usepackage[latin1]{inputenc}
\usepackage[english]{babel}
\usepackage{amsmath}
\usepackage{amsfonts}
\usepackage{amssymb}
\usepackage{epsfig}
\usepackage{amsopn}
\usepackage{amsthm}
\usepackage{color}
\usepackage{graphicx}
\usepackage{enumerate}
\usepackage{mathrsfs}
\usepackage{cite}
\newtheorem{theorem}{Theorem}[section]
\newtheorem{corollary}[theorem]{Corollary}
\newtheorem{lemma}[theorem]{Lemma}
\newtheorem{proposition}[theorem]{Proposition}
\newtheorem{definition}[theorem]{Definition}

\newtheorem*{theorem*}{Theorem}
\newtheorem*{lemma*}{Lemma}
\newtheorem*{remark*}{Remark}
\newtheorem*{definition*}{Definition}
\newtheorem*{proposition*}{Proposition}
\newtheorem*{corollary*}{Corollary}
\numberwithin{equation}{section}
\newcommand{\real}{\mathbb{R}}

\let\ced=\c         

\def\qed{\,\unskip\kern 6pt \penalty 500
\raise -2pt\hbox{\vrule \vbox to8pt{\hrule width 6pt
\vfill\hrule}\vrule}\par}
\definecolor{darkblue}{rgb}{0.05, .05, .65}
\definecolor{darkgreen}{rgb}{0.1, .65, .1}
\definecolor{darkred}{rgb}{0.8,0,0}
\newcommand{\beqn}{\begin{equation}}
\newcommand{\eeqn}{\end{equation}}
\newcommand{\bear}{\begin{eqnarray}}
\newcommand{\eear}{\end{eqnarray}}
\newcommand{\bean}{\begin{eqnarray*}}
\newcommand{\eean}{\end{eqnarray*}}
\begin{document}

\title{\huge \bf Convergence to self-similarity for a degenerate parabolic equation with fast-growing spatially-dependent absorption}
\author{
\Large Razvan Gabriel Iagar\,\footnote{Departamento de Matem\'{a}tica
Aplicada, Ciencia e Ingenieria de los Materiales y Tecnologia
Electr\'onica, Universidad Rey Juan Carlos, M\'{o}stoles,
28933, Madrid, Spain, \textit{e-mail:} razvan.iagar@urjc.es}
\\[4pt] \Large Philippe Lauren\ced{c}ot\,\footnote{Laboratoire de Math\'ematiques (LAMA) UMR 5217, Universit\'e Savoie Mont Blanc, CNRS, F-73000, Chamb\'ery France. \textit{e-mail:} philippe.laurencot@univ-smb.fr}\\ [4pt] }
\date{\today}
\maketitle

\begin{abstract}
The large time behavior of non-negative solutions to the absorption-diffusion equation
\begin{equation*}
	\partial_tu=\Delta u^m-|x|^{\sigma}u^m, \quad (t,x)\in(0,\infty)\times\mathbb{R}^N,
\end{equation*}
with $m>1$ and $\sigma>\sigma_0:=N(m-1)/(m+1)$ is identified. It is shown that all solutions approach a unique stationary solution in self-similar variables, which also provides a universal upper bound (\textit{friendly giant}), strongly contrasting to the standard case $\sigma=0$. On the one hand, the convergence proof exploits the variational structure of the equation and a suitable Caffarelli-Kohn-Nirenberg inequality, along with the B\'enilan-Crandall homogeneity regularizing effect. On the other hand, the detailed study of the stationary problem combines elliptic estimates, Moser iteration and techniques from ordinary differential equations.
\end{abstract}

\smallskip


\smallskip


\section{Introduction}

We aim at describing the large time behavior of non-negative solutions to the Cauchy problem
\begin{subequations}\label{CP}
\begin{equation}\label{eq1}
	\partial_t u=\Delta u^m-|x|^{\sigma}u^m, \quad (t,x)\in(0,\infty)\times\mathbb{R}^N,
\end{equation}
\begin{equation}\label{ic}
	u(0,x)=u_0(x), \quad x\in\mathbb{R}^N,
\end{equation}
\end{subequations}
in the range of exponents
\begin{equation}\label{range.exp}
	m>1, \quad \sigma>\sigma_0:=\frac{N(m-1)}{m+1}.
\end{equation}
Eq.~\eqref{eq1} features a competition between a degenerate diffusion term and an absorption term, the latter being weighted on its turn by an unbounded function enforcing the effect of the absorption in regions where $|x|$ is large. This competition gives rise to a number of interesting mathematical results related to the dynamics of~\eqref{eq1}, as a particular and critical case of the more general class of absorption-diffusion equations
\begin{equation}\label{eq.gen}
	\partial_tu=\Delta u^m-|x|^{\sigma}u^p, \quad (t,x)\in(0,\infty)\times\mathbb{R}^N,
\end{equation}
with $(p,m,\sigma)\in (0,\infty)^2\times \big(\max\{-2,-N\},\infty\big)$. On the one hand, the range of exponents $p>m>1$ and $\sigma\geq 0$ is nowadays well-understood and its analysis led to the discovery of many interesting mathematical objects and features, such as very singular self-similar solutions, asymptotic simplification in the ranges of exponents where the diffusion term governs the long term behavior, or a dynamics featuring a logarithmic time scale in some critical cases, see for example \cite{GV91, KP86, KU87, KV88, Le97, PT86} among the many classical works considering this range for $\sigma=0$. The large time behavior of solutions in this range has been recently generalized to exponents $\sigma>0$ in \cite{IM25,IM26}. On the other hand, the range of exponents $1<p<m$ is less understood and, in our opinion, mathematically more challenging. For $p\in(1,m)$ and $\sigma=0$, Eq.~\eqref{eq.gen} is studied in \cite{MPV91, CV96, CV99}. The description of the large time behavior of its solutions is more involved and, for compactly supported initial conditions, includes the formation of a thin boundary layer, a rigorous proof being only available (to the best of our knowledge) in dimension $N=1$ \cite{CV99}.

In fact, a significant difference between the two ranges $p>m>1$ and $p\in(1,m)$ in Eq.~\eqref{eq.gen} is uncovered in \cite{PT85} (still in dimension $N=1$) with more general weights including $x\mapsto |x|^{\sigma}$ for $\sigma>0$. Specifically, for $p\geq m>1$, any solution to the Cauchy problem~\eqref{CP} with a non-negative compactly supported initial condition enjoys the property of \emph{positivity}; that is, for any $x\in\mathbb{R}$, there is $t_x\in(0,\infty)$ such that $u(t,x)>0$ for any $t>t_x$. In contrast, when $p\in(1,m)$, solutions to~\eqref{CP} with non-negative compactly supported initial condition have the property of \emph{localization} of their supports; that is, there exists $R_0>0$ (depending on $u_0$) such that $u(t)\equiv0$ for any $t>0$ and $x\in\mathbb{R}\setminus(-R_0,R_0)$. It is expected that these properties remain true for the Cauchy problem~\eqref{CP} in higher space dimensions $N\ge 2$.

According to the above discussion, the case $p=m>1$ appears to be critical, separating the scaling behavior for $p>m>1$ and the multiscale behavior for $p\in(1,m)$. Though nothing much is yet known when $p=m>1$, even for $\sigma=0$, deviation from self-similar behavior is observed when $p=m>1$, $\sigma=0$ and $N=1$. Indeed, in that case, special solutions to Eq.~\eqref{eq1} have the form
\begin{equation}\label{waves}
	u(t,x)=t^{-1/(m-1)}f(x+\beta\ln\,t), \quad \beta\geq\frac{m}{m-1},
\end{equation}
as indicated in \cite{CVW97}, where the role of such solutions for the general theory of the equation is investigated. Moreover, the change of variables
\begin{equation*}
	u(t,x)=(1+(m-1)t)^{-1/(m-1)}v(\tau,x), \quad \tau=\frac{1}{m-1}\ln(1+(m-1)t),
\end{equation*}
transforms Eq.~\eqref{eq1} into the Fisher-KPP type equation
\begin{equation}\label{FKPP}
	\partial_{\tau}v=\Delta v^m-v^m+v, \quad (\tau,x)\in(0,\infty)\times\mathbb{R}^N,
\end{equation}
for which the solutions approach traveling waves as $\tau\to\infty$, according to \cite{KR04, G20}. Let us remark here that, by undoing the previous change of variables, traveling wave solutions to~\eqref{FKPP} are mapped into solutions in the form~\eqref{waves}, proving thus that solutions in the form~\eqref{waves} are the asymptotic profiles as $t\to\infty$ of general solutions to the Cauchy problem~\eqref{CP} when $p=m>1$, $\sigma=0$ and $N=1$.

The main goal of this work is then to show that such a non-scaling behavior is not valid for the whole range $\sigma\in (0,\infty)$ and that, when $\sigma$ is sufficiently large, the large time behavior of non-negative solutions to the Cauchy problem~\eqref{CP} is of self-similar nature, strongly departing from the previously discussed profiles valid for $\sigma=0$. We begin with specifying the functional setting we shall work with, which is intimately connected with the variational structure of~\eqref{eq1}.


\paragraph*{Notation.} Throughout the paper, we assume that the initial condition $u_0$ satisfies the following properties:
\begin{equation}\label{icond}
	u_0\in\mathcal{X}_+:= \left\{z\in L_+^{m+1}(\mathbb{R}^N): z^m\in\dot{H}^1(\mathbb{R}^N)\cap L^2_{\sigma}(\mathbb{R}^N)\right\},
\end{equation}
where $\mathcal{X} := \left\{z\in L^{m+1}(\mathbb{R}^N)\ :\ |z|^{m-1}z\in\dot{H}^1(\mathbb{R}^N)\cap L^2_{\sigma}(\mathbb{R}^N)\right\}$ and
\begin{equation*}
	L^{m+1}_+(\mathbb{R}^N) := \{z\in L^{m+1}(\mathbb{R}^N): z\geq0 \;\text{ a.e. in }\; \mathbb{R}^N\},
\end{equation*}
the homogeneous Sobolev space $\dot{H}^1(\mathbb{R}^N)$ denotes the closure of $C_c^\infty(\mathbb{R}^N)$ for the norm $\|z\|_{\dot{H}^1}=\|\nabla z\|_2$ and $L^2_{\sigma}(\mathbb{R}^N) := L^2(\mathbb{R}^N,|x|^\sigma\, dx)$. For $z\in L^2_{\sigma}(\mathbb{R}^N)$, we set
\begin{equation*}
	N_{\sigma}(z) := \left(\int_{\mathbb{R}^N}|x|^{\sigma}z^2(x)\,dx\right)^{1/2}.
\end{equation*}
In addition, we define
\begin{equation*}
\begin{split}
	\mathcal{Y} & := \left\{ z \in L^{(m+1)/m}(\mathbb{R}^N)\ :\ z\in\dot{H}^1(\mathbb{R}^N)\cap L^2_{\sigma}(\mathbb{R}^N)\right\}, \\
	\mathcal{Y}_+ & := \left\{ z \in L_+^{(m+1)/m}(\mathbb{R}^N)\ :\ z\in\dot{H}^1(\mathbb{R}^N)\cap L^2_{\sigma}(\mathbb{R}^N)\right\},
\end{split}
\end{equation*}
and observe that $z\in\mathcal{Y}$ if and only if $|z|^{m-1}z \in\mathcal{X}$.

\paragraph{Main results.} In order to state our main results, we introduce the self-similar or scaling variables and the rescaled function
\begin{equation}\label{resc.var}
	v(s,x)=[1+(m-1)t]^{1/(m-1)} u(t,x), \quad s=\frac{1}{m-1}\ln[1+(m-1)t],
\end{equation}
for $(t,x)\in [0,\infty]\times\mathbb{R}^N$. Equivalently, for $(t,s,x)\in [0,\infty)^2\times\mathbb{R}^N$,
\begin{align*}
	u(t,x) & = [1+(m-1)t]^{-1/(m-1)} v\left( \frac{\ln[1+(m-1)t]}{m-1},x \right), \\
	v(s,x) & = e^s u\left( \frac{e^{(m-1)s}-1}{m-1},x\right).
\end{align*}
A straightforward calculation gives that Eq.~\eqref{eq1} is transformed by this rescaling into
\begin{subequations}\label{CP2}
\begin{equation}\label{eq2}
	\partial_s v(s,x) = \Delta v^m(s,x)-|x|^{\sigma}v^m(s,x)+v(s,x), \quad (s,x)\in(0,\infty)\times\mathbb{R}^N,
\end{equation}
while the initial condition~\eqref{ic} is preserved
\begin{equation}\label{ic2}
	v(0,x)=u_0(x), \quad x\in\mathbb{R}^N.
\end{equation}
\end{subequations}
As usual, existence and stability of self-similar solutions to~\eqref{eq1} are transformed to existence and stability of stationary solutions to~\eqref{eq2} and it is more convenient to state the results and work with the latter. The first result deals with the existence and uniqueness of a stationary solution to Eq.~\eqref{eq2} in $\mathcal{X}_+$.

\begin{theorem}\label{th.uniq}
Let $m$ and $\sigma$ be as in~\eqref{range.exp}. There exists a unique non-negative and non-trivial stationary solution $\varphi\in\mathcal{X}_+$ to Eq.~\eqref{eq2}. Moreover,
\begin{itemize}
\item [$\triangleright$] $\varphi$ is positive on $\mathbb{R}^N$ and belongs to $L^{\infty}(\mathbb{R}^N)\cap C^{2,\alpha}(\mathbb{R}^N)$ for any $\alpha\in (0,1/m)$;
\item [$\triangleright$] $\varphi$ is radially symmetric with non-increasing profile and
\begin{equation}\label{beh.stat}
	\lim\limits_{|x|\to\infty}|x|^{\sigma/(m-1)}\varphi(x)=1.
\end{equation}
\end{itemize} 
\end{theorem}

It obviously follows from~\eqref{resc.var} that the statement of Theorem~\ref{th.uniq} is equivalent to the existence and uniqueness of a classical separate variables solution
\begin{equation}\label{sep.var}
	U(t,x)=[1+(m-1)t]^{-1/(m-1)}\varphi(x), \qquad (t,x)\in [0,\infty)\times\mathbb{R}^N,
\end{equation}
to Eq.~\eqref{eq1}.

The proof of Theorem~\ref{th.uniq} involves several steps and combines the variational structure of~\eqref{eq2} to establish existence and tools from the theory of ordinary differential equations to obtain uniqueness of a non-negative and non-trivial radially symmetric stationary solution to~\eqref{eq2} with non-increasing profile. Getting rid of the symmetry and monotonicity assumptions as stated in Theorem~\ref{th.uniq} is achieved through a detailed study of the dynamics of the Cauchy problem~\eqref{CP2}, which we report now.


\begin{theorem}\label{th.conv}
Let $m$ and $\sigma$ as in \eqref{range.exp} and let $u_0$ be an initial condition satisfying~\eqref{icond}. Then the Cauchy problem~\eqref{CP2} is well-posed in the space $L^\infty_+(\mathbb{R}^N)\cap \mathcal{X}$ and
\begin{equation}\label{asympt.conv2}
	\lim\limits_{s\to\infty}\|v(s)-\varphi\|_{r}=0, \quad r\in[m+1,\infty],
\end{equation}
where $\varphi$ is the unique stationary solution to Eq.~\eqref{eq2}. Furthermore,
\begin{equation*}
	v(s,x) \le \left( \frac{e^{(m-1)s}}{e^{(m-1)s} - 1} \right)^{1/(m-1)} \varphi(x), \quad (s,x)\in (0,\infty)\times \mathbb{R}^N. 
\end{equation*}
\end{theorem}

The following immediate consequence of Theorem~\ref{th.conv} expresses the asymptotic convergence in terms of the initial variables of Eq.~\eqref{eq1}.

\begin{corollary}\label{cor.conv}
In the same conditions as in Theorem~\ref{th.conv}, if $u$ is the solution to the Cauchy problem~\eqref{CP}, then
\begin{equation*}
	\lim\limits_{t\to\infty}t^{1/(m-1)}\|u(t)-U(t)\|_{r}=0, \quad r\in[m+1,\infty],
\end{equation*}
where $U$ is the separate variables solution defined in~\eqref{sep.var}. Furthermore,
\begin{equation*}
	u(t,x) \le \varphi(x) [(m-1)t]^{-1/(m-1)}, \quad (t,x)\in (0,\infty)\times \mathbb{R}^N. 
\end{equation*}
\end{corollary}

In other words, $(t,x)\mapsto [(m-1)t]^{-1/(m-1)} \varphi(x)$ is a so-called \textsl{friendly giant} for~\eqref{CP}, see \cite{AP81} and  \cite[Section~4.2]{DK07} for a similar result for the Cauchy-Dirichlet problem for the porous medium equation in a bounded domain, as well as \cite{MV94} for doubly nonlinear equations and \cite{CLS89} for Hamilton-Jacobi equations, still for the Cauchy-Dirichlet problem in a bounded domain. The existence of friendly giants for Cauchy problems is less customary and we refer to \cite{IM26}, besides Corollary~\ref{cor.conv}, for another result in that direction. Intuitively, due to the fast growth of the weight $|x|^\sigma$ as $|x|\to\infty$, the absorption term acts somehow as a homogeneous Dirichlet boundary condition at infinity.

We emphasize here again that the result in Corollary~\ref{cor.conv} is an unexpected effect of the presence of the weight $|x|^{\sigma}$ for $\sigma$ large enough, as it strongly departs from the convergence to solutions arising from traveling waves established for $\sigma=0$, as commented in the Introduction.

Let us now describe more precisely the main steps of the proofs of Theorems~\ref{th.uniq} and~\ref{th.conv}, which are intertwined, as we shall see. Concerning the former, a classical argument ensures that the functional
\begin{equation*}
	\mathcal{I}(z):= \frac{1}{2}\big( \|\nabla z^m\|_2^2 + N_{\sigma}^2(z^m) \big) - \frac{m}{m+1} \|z\|_{m+1}^{m+1}, \quad z\in\mathcal{X}_+,
\end{equation*}
which is well-defined on $\mathcal{X}_+$, is a Lyapunov functional for~\eqref{CP2}, while any non-negative stationary solution $v_*$ to~\eqref{eq2} is connected to a critical point $V_*$ of the functional
\begin{equation*}
	\mathcal{J}(z) = \frac{\|\nabla z\|_2^2}{2} + \frac{N_\sigma^2(z)}{2} - \frac{m}{m+1} \|z\|_{(m+1)/m}^{(m+1)/m}, \quad z\in \mathcal{Y},
\end{equation*}
by the identity $v_*=V_*^{1/m}$, due to the obvious identity $\mathcal{J}(z) = \mathcal{I}\big(z^{1/m}\big)$ for $z\in\mathcal{Y}_+$. We shall thus prove the existence of a stationary solution to~\eqref{eq2} by showing that $\mathcal{J}$ is bounded from below on $\mathcal{Y}$ and has at least a non-negative and non-trivial radially symmetric minimizer with non-increasing profile. The uniqueness part of Theorem~\ref{th.uniq} is very important for the rest of the analysis and is split into two steps. We first prove the uniqueness under the assumption of radial symmetry and monotonicity and thereby establish that Eq.~\eqref{eq2} has a unique non-trivial radially symmetric stationary solution $\varphi\in\mathcal{X}_+$ with non-increasing profile. We next show the stability of $\varphi$ for the dynamics of~\eqref{CP2} restricted to non-negative radially symmetric initial conditions with non-increasing profiles. In a second step, (parabolic) comparison arguments are used to established that the basin of attraction of $\varphi$ actually includes a broader class of initial data and in particular any non-negative and non-trivial stationary solution to~\eqref{eq2}. The proof of this last property also relies on (elliptic) comparison arguments, which are used to establish that any critical point $V$ of $\mathcal{J}$ in $\mathcal{Y}_+$ is actually a bounded and positive classical solution to the elliptic equation
\begin{equation}
	- \Delta V(x) + |x|^\sigma V(x) - V^{1/m}(x) = 0, \quad x\in\mathbb{R}^N, \label{EE}
\end{equation}
which satisfies the tail estimate $V(x)\le C |x|^{-m\sigma/(m-1)}$, $x\in\mathbb{R}^N$, for some $C>0$.

The optimality of the range of $\sigma$ for which Theorem~\ref{th.uniq} holds true is guaranteed by the next result, which is a consequence of a Pohozaev identity.

\begin{proposition}\label{prop.ne}
	For $\sigma\in[0,\sigma_0]$, there is no non-trivial stationary solution to Eq.~\eqref{eq2} belonging to $\mathcal{X}$.
\end{proposition}

It is actually likely that the elliptic equation~\eqref{EE} has non-negative solutions for $\sigma\in [0,\sigma_0]$ but they do not belong to $\mathcal{X}$ and are thus excluded from the variational framework associated to~\eqref{CP2}.

\medskip

The next section is devoted to the well-posedness of~\eqref{CP} and~\eqref{CP2}, along with some properties of non-negative solutions including lower and upper bounds and the availability of a Lyapunov functional for~\eqref{CP2}. We also recall in Section~\ref{sec.aux} the Caffarelli-Kohn-Nirenberg (CKN) inequalities, which provide a lower bound for the above mentioned Lyapunov functional. Section~\ref{sec.stso} focuses on the stationary solutions to Eq.~\eqref{eq2}, which we analyze through the equivalent formulation~\eqref{EE}. This section first gathers properties of generic stationary solutions, including regularity, boundedness and tail estimates. The second part of Section~\ref{sec.stso} is restricted to radially symmetric stationary solutions with non-increasing profile and culminates actually in the uniqueness of such a stationary solution, proving Theorem~\ref{th.uniq} in this specific setting. The dynamics of~\eqref{CP2} is studied in Section~\ref{sec.cv} and we start with the compactness of the trajectories, along with the identification of their $\omega$-limit set. All the information obtained so far is then unified to allow us to prove Theorem~\ref{th.conv} and to complete the proof of Theorem~\ref{th.uniq}. The final section is devoted to the non-existence of stationary solutions to~\eqref{EE} belonging to $\mathcal{Y}_+$ for $\sigma\in [0,\sigma_0]$.

\section{Auxiliary results}\label{sec.aux}

We gather in this section a number of preparatory results related to Eq.~\eqref{eq1} and Eq.~\eqref{eq2}, concerning their well-posedness in $L_+^{\infty}(\mathbb{R}^N)$ and in $\mathcal{X}_+$, the definition of the Lyapunov functional and an optimal CKN inequality. Some of these results have independent interest, besides being employed in the proofs of Theorem~\ref{th.uniq} and Theorem~\ref{th.conv}. For simplicity, the section is divided into short specialized subsections.

\subsection{Well-posedness in $L^{\infty}(\mathbb{R}^N)$}

We recall a general well-posedness result, valid for initial conditions $u_0\in L^{\infty}_{+}(\mathbb{R}^N)$. To this end, we first introduce the notion of weak solution as follows:

\begin{definition}\label{def.ws}
A non-negative weak solution to the Cauchy problem~\eqref{CP} is a function $u\in L_+^\infty((0,\infty)\times\mathbb{R}^N)$ such that, for any $T>0$,
\begin{equation*}
	u^m \in L^2\big((0,T),H^1_{\text{loc}}(\mathbb{R}^N)\big)
\end{equation*}
and it satisfies the weak formulation
\begin{equation}\label{weak.sol}
	\int_0^T \int_{\mathbb{R}^N} \Big[(u_0-u) \partial_t \zeta + \nabla u^m \cdot \nabla\zeta + |x|^\sigma u^{m} \zeta \Big]\ dxds = 0
\end{equation}
for all $\zeta\in C_c^1([0,T)\times\mathbb{R}^N)$.
\end{definition}

The following result establishes the well-posedness and the comparison principle for solutions to the Cauchy problem~\eqref{CP} with bounded and non-negative initial conditions.

\begin{proposition}\label{prop.wp}
Let $m>1$, $\sigma>0$ and $u_0\in L^{\infty}_{+}(\mathbb{R}^N)$. Then there is a unique non-negative weak solution to the Cauchy problem~\eqref{CP} which satisfies
\begin{equation}\label{wp0}
	\|u(t)\|_\infty \le \|u_0\|_\infty\,, \qquad t\ge 0.
\end{equation}
Furthermore, the following comparison principle holds true: given $u_{0,i}\in L_+^\infty(\mathbb{R}^N)$, $i=1,2$, such that $u_{0,1}\le u_{0,2}$ in $\mathbb{R}^N$, the corresponding non-negative weak solutions $u_1$ and $u_2$ to the Cauchy problem~\eqref{CP} satisfy $u_1\le u_2$ in $(0,\infty)\times\mathbb{R}^N$.
\end{proposition}

Proposition~\ref{prop.wp} is just a particular case (letting $q=m$ in the notation therein) of a more general result given in \cite[Section~2]{ILS24}. The following standard property will be very useful in the proof of Theorem~\ref{th.conv}.

\begin{lemma}\label{lem.rad}
Let $u_0\in L^{\infty}_+(\mathbb{R}^N)$ be radially symmetric with non-increasing profile in radial variables and $u$ be the solution to the Cauchy problem~\eqref{CP}. Then $u(t)$ is radially symmetric with non-increasing profile for any $t\in(0,\infty)$.
\end{lemma}

\begin{proof}
The radial symmetry of $u(t)$ for $t>0$ is obvious from the rotational invariance of Eq.~\eqref{eq1}. As for the monotonicity of the profile, we provide a formal proof, which can be made rigorous by an approximation argument as usual. Introducing the radially symmetric variable $r=|x|$ and setting $u(r)=u(|x|)=u(x)$ for simplicity, we derive the equation solved by $\partial_r u$. To this end, we start from Eq.~\eqref{eq1} written in radially symmetric variables as
\begin{equation}\label{eq1.rad}
	\partial_t u(t,r) = \partial_r^2 u^m(t,r) + \frac{N-1}{r} \partial_r u^m(t,r) - r^{\sigma} u^m(t,r), \quad (t,r)\in (0,\infty)^2,
\end{equation}
and differentiate~\eqref{eq1.rad} with respect to $r$ to find, by direct calculation, that $w:=\partial_r u$ solves
\begin{equation}\label{eq1.deriv}
	\begin{split}
		\partial_t w &=mu^{m-1} \partial_r^2 w + 3m(m-1)u^{m-2}w\partial_r w + m(m-1)(m-2)u^{m-3}w^3\\
		& \quad -\frac{m(N-1)}{r^2}u^{m-1}w + \frac{m(N-1)(m-1)}{r}u^{m-2}w^2\\
		& \quad + \frac{m(N-1)}{r}u^{m-1} \partial_r w - mr^{\sigma}u^{m-1}w - \sigma r^{\sigma-1}u^m.
	\end{split}
\end{equation}
This is a degenerate parabolic equation allowing for a comparison principle on the positivity set of $u$. Since $w\equiv0$ is a strict supersolution to~\eqref{eq1.deriv}, we infer by comparing on the positivity set of $u$ that $w(t)\geq0$ for any $t\geq0$, provided $w_0=\partial_r u_0\leq0$, which completes the proof.
\end{proof}

\subsection{Lower and upper bounds}

We supplement Proposition~\ref{prop.wp} with the following general lower bound for non-negative solutions to Eq.~\eqref{eq1}, which will be essential in the sequel for showing that the limit of the corresponding solution $v(s)$ to Eq.~\eqref{eq2} as $s\to\infty$ is not the zero function.

\begin{lemma}\label{lem.bc}
Let $m>1$, $\sigma>0$ and $u_0\in L^{\infty}_{+}(\mathbb{R}^N)$. Then the corresponding	weak solution $u$ to Eq.~\eqref{CP} satisfies
\begin{equation}
	t^{1/(m-1)} u(t,x) \ge t_0^{1/(m-1)} u(t_0,x), \qquad 0<t_0<t, \ x\in\mathbb{R}^N. \label{tmu}
\end{equation}
Moreover, if $u_0\not\equiv 0$, then $u(t)\not\equiv 0$ for all $t>0$.
\end{lemma}

\begin{proof}
We first recall that the homogeneity regularizing effect derived in the classical paper \cite[Theorem~2]{BC81} implies that the non-negative weak solution $u$ to Eq.~\eqref{CP} satisfies
\begin{equation}\label{bcest}
	\partial_t u(t,x) \geq - \frac{u(t,x)}{(m-1)t}, \quad (t,x)\in(0,\infty)\times\mathbb{R}^N,
\end{equation}
in the sense of distributions, from which we deduce that the map $t\mapsto t^{1/(m-1)} u(t,x)$ is a non-decreasing function of time for all $x\in\mathbb{R}^N$, whence~\eqref{tmu}.

We next infer from~\eqref{weak.sol} that, for $t>0$ and $\zeta\in C_c^1(\mathbb{R}^N)$,
\begin{align*}
	\int_{\mathbb{R}^N} \zeta(x) u(t,x)\ dx & = \int_{\mathbb{R}^N} \zeta(x) u_0(x)\ dx \\
	& \quad - \int_0^t \int_{\mathbb{R}^N} \left[ \nabla \zeta(x)\cdot \nabla u^m(s,x) + \zeta(x) |x|^\sigma u^m(s,x) \right]\ dxds.
\end{align*}
Since $u^m\in L^2\big((0,t),H^1_{\text{loc}}(\mathbb{R}^N)\big)$ and $(s,x)\mapsto |x|^\sigma u^m(s,x)\in L^2\big((0,t),L^2_{\text{loc}}(\mathbb{R}^N)\big)$, we may let $t\to 0$ in the above identity and deduce from the Lebesgue dominated convergence theorem that
\begin{equation*}
	\lim_{t\to 0} \int_{\mathbb{R}^N} \zeta(x) u(t,x)\ dx = \int_{\mathbb{R}^N} \zeta(x) u_0(x)\ dx.
\end{equation*}
A density argument then implies that, for any $R>0$,
\begin{equation}
	\lim_{t\to 0} \int_{B(0,R)} u(t,x)\ dx = \int_{B(0,R)} u_0(x)\ dx. \label{x10}
\end{equation}
In particular, since $u_0\not\equiv 0$, there is $R_0>0$ large enough such that
\begin{equation*}
	\int_{B(0,R_0)} u_0(x)\ dx > 0
\end{equation*}
and we infer from~\eqref{x10} that there is $T_0>0$ such that
\begin{equation*}
	\int_{B(0,R_0)} u(t,x) \ge \frac{1}{2} \int_{B(0,R_0)} u_0(x)\ dx, \quad t\in (0,T_0).
\end{equation*}
In other words, $u(t)\not\equiv 0$ for $t\in (0,T_0)$, a property which extends to all positive times due to~\eqref{tmu}.
\end{proof}

We next adapt an argument from \cite[Proposition~3.1]{CdPE98} to derive a refined $L^\infty$-estimate on $u$, which improves~\eqref{wp0}.
	
\begin{lemma}\label{lem.lub}
Let $m>1$, $\sigma>0$ and $u_0\in L^{\infty}_{+}(\mathbb{R}^N)$. Then the corresponding non-negative weak solution $u$ to the Cauchy problem~\eqref{CP} satisfies
\begin{equation*}
	\|u(t)\|_\infty^{N(m-1)+2(m+1)} \le C_1 \|u(t)\|_{m+1}^{2(m+1)} t^{-N}\,, \qquad t\ge 0,
\end{equation*}
for some constant $C_1>0$ depending only on $N$ and $m$.
\end{lemma}

\begin{proof}
Let $t>0$ and $x_0\in\mathbb{R}^N$. We infer from Eq.~\eqref{eq1} and~\eqref{bcest} that
\begin{equation*}
	\Delta u^m(t,x) = \partial_t u(t,x) + |x|^\sigma u^m(t,x) \ge - \frac{u(t,x)}{(m-1)t} \ge - \frac{\|u(t)\|_\infty}{(m-1)t}, \quad x\in\mathbb{R}^N,
\end{equation*}
from which we deduce that
\begin{equation*}
	\Delta \left[ u^m(t,x) + \frac{\|u(t)\|_\infty}{(m-1)t} \frac{|x-x_0|^2}{2N} \right] \ge 0, \quad x\in \mathbb{R}^N.
\end{equation*}
Therefore, $x\mapsto u^m(t,x) + \|u(t)\|_\infty |x-x_0|^2/(2N(m-1)t)$ is a subharmonic function in $\mathbb{R}^N$ and the mean-value theorem entails that, for $R>0$,
\begin{equation*}
	u^m(t,x_0) \le \frac{1}{\varpi_N R^N} \int_{B(x_0,R)} \left[ u^m(t,x) + \frac{\|u(t)\|_\infty}{(m-1)t} \frac{|x-x_0|^2}{2N} \right]\ dx,
\end{equation*}
where $\varpi_N$ denotes the volume of the unit ball in $\mathbb{R}^N$, see \cite[Theorem~2.1]{GT01}. We then infer from H\"older's inequality that
\begin{align*}
	u^m(t,x_0) & \le \big(\varpi_N R^N\big)^{-m/(m+1)} \left( \int_{B(0,R)} u^{m+1}(t,x)\ dx \right)^{m/(m+1)} \\
	& \quad + \frac{\|u(t)\|_\infty}{2(m-1)R^N t} \int_0^R r^{N+1}\ dr \\
	& \le \big(\varpi_N R^N\big)^{-m/(m+1)} \|u(t)\|_{m+1}^m + \frac{\|u(t)\|_\infty}{(m-1)t} R^2.
\end{align*}
We optimize with respect to $R$ by choosing
\begin{equation*}
	R=R(t) = \left[ \frac{t\|u(t)\|_{m+1}^m}{\|u(t)\|_\infty} \right]^{(m+1)/[mN+2(m+1)]}
\end{equation*}
in the above inequality to obtain
\begin{equation*}
	u^m(t,x_0) \le \left[ C_1^m \frac{\|u(t)\|_{m+1}^{2m(m+1)} \|u(t)\|_\infty^{Nm}}{t^{Nm}} \right]^{1/[Nm+2(m+1)]},
\end{equation*}
with
\begin{equation*}
	C_1 := \left( \frac{1}{\varpi_N^{m/(m+1)}} + \frac{1}{m-1} \right)^{[Nm+2(m+1)]/m}.
\end{equation*}
The above inequality being valid for all $x_0\in\mathbb{R}^N$, we further obtain
\begin{equation*}
	\|u(t)\|_\infty^m \le \left[ C_1^m \frac{\|u(t)\|_{m+1}^{2m(m+1)} \|u(t)\|_\infty^{Nm}}{t^{Nm}} \right]^{1/[Nm+2(m+1)]},
\end{equation*}
whence
\begin{equation*}
	\|u(t)\|_\infty^{N(m-1)+2(m+1)} \le C_1 \|u(t)\|_{m+1}^{2(m+1)} t^{-N},
\end{equation*}
as claimed.
\end{proof}
The results in Lemmas ~\ref{lem.bc} and~\ref{lem.lub} readily translate to the corresponding solution $v$ of Eq.~\eqref{eq2} by recalling that $v(0,x)=u_0(x)$ for $x\in\mathbb{R}^N$ and by applying~\eqref{resc.var}. We gather them in the next proposition for later use.

\begin{proposition}\label{prop.v}
Let $m>1$, $\sigma>0$ and $u_0\in L_+^\infty(\mathbb{R}^N)$. The corresponding solution $v$ to the Cauchy problem~\eqref{CP2} satisfies
\begin{equation}
	\left( \frac{e^{(m-1)s} - 1}{e^{(m-1)s}} \right)^{1/(m-1)} v(s,x) \ge \left( \frac{e^{(m-1)s_0} - 1}{e^{(m-1)s_0}} \right)^{1/(m-1)} v(s_0,x), \quad x\in\mathbb{R}^N \label{x11}
\end{equation}
for $s>s_0>0$ and
\begin{equation}
	\|v(s)\|_\infty^{N(m-1)+2(m+1)} \le C_1 (m-1)^N \left( \frac{e^{(m-1)s}}{e^{(m-1)s}-1} \right)^N \|v(s)\|_{m+1}^{2(m+1)}, \quad s\ge 0. \label{x12}
\end{equation}
Furthermore, if $u_0\not\equiv 0$, then $v(s)\not\equiv 0$ for any $s>0$.
\end{proposition}

\subsection{A sharp Caffarelli-Kohn-Nirenberg inequality}\label{subsec.CKN}

Another fundamental tool in the forthcoming proofs consists in the celebrated CKN inequalities. We first recall below, for the reader's convenience, the general form and conditions of the CKN inequalities. Let $(q_1, q_2, q_3, \gamma_1, \gamma_2, \gamma_3, a)\in\mathbb{R}^7$ be such that
\begin{subequations}\label{CKN.cond}
\begin{equation}\label{CKN.cond1}
q_1>0, \quad q_2\geq1, \quad q_3>0, \quad a\in[0,1],
\end{equation}
\begin{equation}\label{CKN.cond2}
\frac{1}{q_1}+\frac{\gamma_1}{N}>0, \quad \frac{1}{q_2}+\frac{\gamma_2}{N}>0, \quad \frac{1}{q_3}+\frac{\gamma_3}{N}>0,
\end{equation}
\begin{equation}\label{CKN.cond3}
\frac{1}{q_1}+\frac{\gamma_1}{N}=a\left(\frac{1}{q_2}+\frac{\gamma_2-1}{N}\right)+(1-a)\left(\frac{1}{q_3}+\frac{\gamma_3}{N}\right),
\end{equation}
\begin{equation}\label{CKN.cond4}
\gamma_1\leq a\gamma_2+(1-a)\gamma_3,
\end{equation}
\begin{equation}\label{CKN.cond5}
\frac{1}{q_1}\leq\frac{a}{q_2}+\frac{1-a}{q_3} \quad {\rm if} \ a=0 \ {\rm or} \ a=1.
\end{equation}
\end{subequations}
In these conditions and notation, we have the following interpolation inequalities,see \cite[Theorem]{CKN1984} and \cite[Theorem~1.2]{LW23}.

\begin{theorem}[CKN inequalities]\label{th.CKN}
Let $(q_1,q_2,q_3,\gamma_1,\gamma_2,\gamma_3,a)\in\mathbb{R}^7$ such that~\eqref{CKN.cond1} and \eqref{CKN.cond2} are satisfied. Then, there exists $C>0$ such that
\begin{equation}\label{CKN}
\||x|^{\gamma_1} z\|_{q_1}\leq C\||x|^{\gamma_2}\nabla z\|_{q_2}^a\||x|^{\gamma_3}z\|_{q_3}^{1-a}
\end{equation}
holds true for all $z\in C_c^1(\mathbb{R}^N)$ if and only if~\eqref{CKN.cond3}, \eqref{CKN.cond4} and~\eqref{CKN.cond5} are satisfied. Moreover, in any compact set in the space of parameters satisfying~\eqref{CKN.cond1} and~\eqref{CKN.cond2}, the constant $C$ is bounded.
\end{theorem}

We next particularize the previous general CKN inequalities to our case of interest.

\begin{proposition}\label{prop.CKN}
Let $m$ and $\sigma$ be as in \eqref{range.exp} and consider $s>0$ satisfying
\begin{equation}\label{lim.CKN}
\frac{m+1}{m} \leq s<2^*:=\begin{cases}
                           \displaystyle{\frac{2N}{N-2}} & \text{if } N\geq3 \\
                           \infty & \text{if } N\in\{1,2\}.
                         \end{cases}
\end{equation}
Then $\dot{H}^1(\mathbb{R}^N)\cap L^2_{\sigma}(\mathbb{R}^N)$ is continuously embedded in $L^s(\mathbb{R}^d)$ and the following CKN inequality
\begin{equation}\label{opt.CKN}
\|z\|_s\leq \kappa_s\|\nabla z\|_{2}^{a}N_{\sigma}^{1-a}(z), \quad a:=\frac{s\sigma+N(s-2)}{(\sigma+2)s}
\end{equation}
holds true for any $z\in\dot{H}^1(\mathbb{R}^N)\cap L^2_{\sigma}(\mathbb{R}^N)$ for some constant $\kappa_s>0$ depending only on $N$, $m$, $\sigma$ and $s$.
\end{proposition}

\begin{proof}
We may apply Theorem~\ref{th.CKN} with
\begin{equation*}
(q_1,q_2,q_3,\gamma_1,\gamma_2,\gamma_3)=\left(s,2,2,0,0,\frac{\sigma}{2}\right).
\end{equation*}
Indeed, with this choice of parameters, the condition~\eqref{CKN.cond2} is obviously satisfied, while~\eqref{CKN.cond3} reads
\begin{equation*}
\begin{split}
	\frac{1}{s}&=a\left(\frac{1}{2}-\frac{1}{N}\right)+(1-a)\left(\frac{1}{2}+\frac{\sigma}{2N}\right)\\
	&=\frac{N+\sigma}{2N}-\frac{a(\sigma+2)}{2N},
\end{split}
\end{equation*}
which readily gives
\begin{equation*}
	a=\frac{\sigma s+N(s-2)}{(\sigma+2)s}.
\end{equation*}
In order to fulfill the condition~\eqref{CKN.cond1}, we are left to check that $a\in[0,1]$, as the other positivity conditions are obvious. On the one hand, we observe that
\begin{equation*}
	1-a=\frac{2N-(N-2)s}{s(\sigma+2)}>0,
\end{equation*}
since $s<2^*$. On the other hand, taking into account that $\sigma>\sigma_0$, we deduce from~\eqref{lim.CKN} that
\begin{equation*}
\begin{split}
	\sigma s + N(s-2) & > \frac{N(m-1)s}{m+1}+N(s-2) = \frac{N}{m+1}[2ms-2(m+1)] \\
	& = \frac{2mN}{m+1}\left(s-\frac{m+1}{m}\right) \geq 0,
\end{split}
\end{equation*}
proving thus that $a>0$. Thus, the condition~\eqref{CKN.cond1} is satisfied and then the choice of the exponents $\gamma_1=\gamma_2=0$ obviously ensures the condition~\eqref{CKN.cond4}. We are thus in a position to complete the proof by applying~\eqref{CKN} and observing that
\begin{equation*}
\||x|^{\gamma_3}z\|_{q_3}=\left(\int_{\mathbb{R}^N}|x|^{\sigma}|z(x)|^2\,dx\right)^{1/2}=N_{\sigma}(z),
\end{equation*}
which leads to \eqref{opt.CKN}.
\end{proof}

The following immediate consequence of the inequality~\eqref{opt.CKN} will be useful later.

\begin{corollary}\label{cor.CKN}
Let $m$ and $\sigma$ be as in \eqref{range.exp} and consider $s>0$ satisfying~\eqref{lim.CKN}. For $z\in \dot{H}^1(\mathbb{R}^N)\cap L^2_{\sigma}(\mathbb{R}^N)$, we have
\begin{equation}\label{CKN-Young}
	\|z\|_s^{2}\leq\kappa_s^2 \big[\|\nabla z\|_2^2+N_{\sigma}^2(z)\big].
\end{equation}
\end{corollary}

\begin{proof}
We infer from~\eqref{opt.CKN} and Young's inequality that
\begin{equation*}
	\|z\|_s^2\leq\kappa_s^2\|\nabla z\|_2^{2a}N_{\sigma}^{2(1-a)}(z)\leq\kappa_s^2 \left[a\|\nabla z\|_2^2+(1-a)N_{\sigma}^2(z)\right],
\end{equation*}
and the inequality \eqref{CKN-Young} follows from the fact that $a\in(0,1)$.
\end{proof}

\subsection{A Lyapunov functional}\label{subsec.Lyap}

For any function $z\in\mathcal{X}$, we introduce the functional
\begin{equation*}
	\mathcal{I}(z):=\frac{1}{2}\big(\|\nabla (|z|^{m-1}z)\|_2^2+N_{\sigma}^2(|z|^m)\big)-\frac{m}{m+1}\|z\|_{m+1}^{m+1}.
\end{equation*}
The first result establishes the boundedness from below of the functional $\mathcal{I}$ and strongly relies on the CKN inequalities derived in Proposition \ref{prop.CKN}.

\begin{lemma}\label{lem.lower}
Let $m$ and $\sigma$ be as in~\eqref{range.exp}. There exists $\mathcal{I}_0 > 0$ depending only on $N$, $m$ and $\sigma$ such that
\begin{equation*}
	\mathcal{I}(z)\geq \frac{1}{4}\big(\|\nabla (|z|^{m-1}z)\|_2^2 + N_{\sigma}^2(|z|^m)\big) - \mathcal{I}_0, \quad z\in\mathcal{X}.
\end{equation*}
\end{lemma}

\begin{proof}
Let $z\in\mathcal{X}$. Then $|z|^{m-1}z\in \dot{H}^1(\mathbb{R}^N)\cap L^2_{\sigma}(\mathbb{R}^N)$ and we infer from~\eqref{CKN-Young} with $|z|^{m-1}z$ instead of $z$ and $s=(m+1)/m$ that
\begin{equation}\label{interm1}
\begin{split}
	\frac{1}{4} \left( \|\nabla (|z|^{m-1}z)\|_{2}^2 + N_{\sigma}(|z|^m)^{2} \right) & \ge \frac{\kappa_{(m+1)/m}^2}{4} \||z|^m\|_{(m+1)/m}^2 \\
	& = \frac{\kappa_{(m+1)/m}^2}{4} \|z\|_{m+1}^{2m}.
\end{split}
\end{equation}
Since $m+1<2m$, Young's inequality entails that there is a positive constant $\mathcal{I}_0$ depending only on $N$, $m$ and $\sigma$ such that
\begin{equation}\label{interm2}
	\frac{m}{m+1} \|z\|_{m+1}^{m+1} \le \frac{\kappa_{(m+1)/m}^2}{4} \|z\|_{m+1}^{2m} + \mathcal{I}_0.
\end{equation}
Combining~\eqref{interm1} and~\eqref{interm2} completes the proof.
\end{proof}

The following result shows that $\mathcal{I}$ is a Lyapunov functional for~\eqref{CP2}, which allows us to prove the well-posedness of~\eqref{CP2} (and thus that of~\eqref{CP} as well) in the energy space $\mathcal{X}$, besides providing valuable information on the long term behavior of solutions to~\eqref{CP2}.

\begin{proposition}\label{prop.Lyap}
Let $m$ and $\sigma$ be as in \eqref{range.exp} and let $v$ be the solution to the Cauchy problem~\eqref{CP2} with initial condition $v_0\in L_+^\infty(\mathbb{R}^N)\cap \mathcal{X}$. Then $v(s)\in L_+^\infty(\mathbb{R}^N)\cap \mathcal{X}$ for all $s\ge 0$ and
\begin{equation}
	\mathcal{I}(v(s)) + \frac{4m}{(m+1)^2} \int_0^s \left\| \partial_s \big(v^{(m+1)/2}\big)(s_*) \right\|_2^2\, ds_* \le \mathcal{I}(u_0), \quad s\ge 0. \label{diss}
\end{equation}
Moreover, there is $C_2(u_0)>0$ depending only $N$, $m$, $\sigma$ and $u_0$ such that
\begin{equation}
	\big\|\nabla v^m(s)\big\|_2 + N_\sigma(v^m(s)) + \|v(s)\|_{m+1} + \|v(s)\|_\infty \le C_2(u_0), \quad s\ge \frac{\ln{m}}{m-1}. \label{ubv}
\end{equation}
\end{proposition}

\begin{proof}
We only provide a formal proof of~\eqref{diss}, classical approximation arguments being needed for a rigorous justification. We multiply Eq.~\eqref{eq2} by $\partial_sv^m$ and integrate over $\mathbb{R}^N$ to obtain
\begin{equation*}
\begin{split}
	0\leq\int_{\mathbb{R}^N} \partial_s v \partial_s v^m\,dx&=\int_{\mathbb{R}^N} \left[\Delta v^m \partial_s v^m - |x|^{\sigma} v^m \partial_s v^m + v \partial_s v^m\right]\,dx\\
	& = \int_{\mathbb{R}^N} \left[-\frac{1}{2} \partial_s|\nabla v^m|^2 - \frac{1}{2} |x|^{\sigma} \partial_s v^{2m} + \frac{m}{m+1} \partial_s v^{m+1}\right]\,dx\\
	&= - \frac{d}{ds} \left( \frac{\|\nabla v^m\|_2^2}{2} + \frac{N_{\sigma}^2(v^m)}{2} - \frac{m}{m+1} \|v\|_{m+1}^{m+1}\right)\\& = - \frac{d}{ds} \mathcal{I}(v),
\end{split}
\end{equation*}
from which~\eqref{diss} follows after integrating with respect to $s$, the above identity becoming an inequality in the approximation process. Now, for $s\ge 0$, we combine Lemma~\ref{lem.lower} and~\eqref{diss} to find that
\begin{equation*}
	\mathcal{I}(u_0) \ge \mathcal{I}(v(s)) \ge \frac{1}{4}\big(\|\nabla v^m(s)\|_2^2 + N_{\sigma}^2(v^m(s))\big) - \mathcal{I}_0,
\end{equation*}
which proves that $v^m(s)$ belongs to $\dot{H}^1(\mathbb{R}^N)\cap L_\sigma^2(\mathbb{R}^N)$ with
\begin{equation}
	\|\nabla v^m(s)\|_2^2 + N_{\sigma}^2(v^m(s)) \le 4 \left[ \mathcal{I}(u_0) + \mathcal{I}_0 \right], \quad s\ge 0. \label{x13}
\end{equation}
We then infer from Corollary~\ref{cor.CKN} (with $s=(m+1)/m$) and~\eqref{x13} that, for $s\ge 0$, $v^m(s)$ belongs to $L^{(m+1)/m}(\mathbb{R}^N)$ with
\begin{align}
	\|v(s)\|_{m+1}^{2m} & = \|v^m(s)\|_{(m+1)/m}^2 \le \kappa_{(m+1)/m}^2 \left[ \|\nabla v^m(s)\|_2^2 + N_\sigma^2(v^m(s)) \right] \nonumber \\
	& \le 4 \kappa_{(m+1)/m}^2 \left[ \mathcal{I}(u_0) + \mathcal{I}_0 \right]. \label{x14}
\end{align}
Therefore, $v(s)\in\mathcal{X}$ for all $s\ge 0$ by~\eqref{x13} and~\eqref{x14} and $v(s)\in L^\infty(\real^N)$ by~\eqref{resc.var} and Proposition~\ref{prop.wp}. Finally, it follows from~\eqref{x12} and~\eqref{x14} that, for $s\ge \ln{m}/(m-1)$,
\begin{align*}
	\|v(s)\|_\infty & \le C_1 (m-1)^N \left( \frac{e^{(m-1)s}}{e^{(m-1)s}-1} \right)^N \|v(s)\|_{m+1}^{2(m+1)} \\
	& \le 4^{(m+1)/m} C_1 m^N \kappa_{(m+1)/m}^{2(m+1)/m} \left[ \mathcal{I}(u_0) + \mathcal{I}_0 \right]^{(m+1)/m}
\end{align*}
which, together with~\eqref{x13} and~\eqref{x14}, entails~\eqref{ubv} and complete the proof.
\end{proof}

We have now all the required preparations in order to complete the proofs of the two main results.

\section{Stationary solutions}\label{sec.stso}

We start with the study of the stationary problem associated to Eq.~\eqref{eq2} and recall that, if $v\in\mathcal{X}_+$ is a stationary solution to~\eqref{eq2}, then $V=v^m\in\mathcal{Y}_+$ is a solution to Eq.~\eqref{EE}; that is,
\begin{equation*}
	V\in \mathcal{Y} = \dot{H}^1(\mathbb{R}^N)\cap L^2_{\sigma}(\mathbb{R}^N)\cap L^{(m+1)/m}(\mathbb{R}^N)
\end{equation*}
is a weak solution to the elliptic equation
\begin{equation}\label{eq.stat}
-\Delta V(x)+|x|^{\sigma}V(x)=V^{1/m}(x), \quad x\in\mathbb{R}^N.
\end{equation}

This section is split into three parts: the first one deals with arbitrary non-trivial solutions $V\in\mathcal{Y}_+$ to Eq.~\eqref{eq.stat}, which are shown to be classical bounded solutions and satisfy a tail estimate, see Section~\ref{sec.br}. We next exploit the variational structure of Eq.~\eqref{eq.stat} and classical properties of rearrangements to construct at least one radially symmetric solution with non-increasing profile, see Section~\ref{sec.exrad}. We end up this section by establishing the uniqueness of such a solution, proved in Section~\ref{sec.uniqrad} using techniques from the theory of ordinary differential equations.

\subsection{Boundedness and regularity}\label{sec.br}

We first show the boundedness of solutions $V\in\mathcal{Y}_+$ to~\eqref{eq.stat}. To this end, we employ a Moser iteration technique based once more on the CKN inequality~\eqref{opt.CKN}, used in its full strength here. In order to prepare the iteration, we first prove the following inequality:

\begin{lemma}\label{lem.Moser}
Let $m$ and $\sigma$ be as in~\eqref{range.exp}, $s\in((m+1)/m,2^*)$, $p\geq2$ and $V\in\mathcal{Y}_+$ be a solution to Eq.~\eqref{eq.stat}. Then there is $C>0$ (depending on $m$, $N$, $s$, $\sigma$ but not on $p$) such that
\begin{equation}\label{est.Moser}
	\left(\int_{\mathbb{R}^N}V^{sp/2}(x)\,dx\right)^{2/s}\leq Cp\int_{\mathbb{R}^N}V^{p-(m-1)/m}(x)\,dx.
\end{equation}
\end{lemma}

\begin{proof}
We multiply Eq.~\eqref{eq.stat} by $V^{p-1}$ and integrate over $\mathbb{R}^N$ to obtain
\begin{equation*}
	\int_{\mathbb{R}^N}\nabla V^{p-1}(x)\cdot\nabla V(x)\,dx+\int_{\mathbb{R}^N}|x|^{\sigma}V^p(x)\,dx=\int_{\mathbb{R}^N}V^{p-1+1/m}(x)\,dx,
\end{equation*}
which is equivalent, after easy manipulations, to
\begin{equation}\label{interm5}
	\frac{4(p-1)}{p^2}\|\nabla V^{p/2}\|_2^2+N_{\sigma}^2(V^{p/2})=\int_{\mathbb{R}^N}V^{p-1+1/m}(x)\,dx.
\end{equation}
Since $p\geq2$, we also have
\begin{equation}\label{interm6}
	\frac{4(p-1)}{p^2}\|\nabla V^{p/2}\|_2^2+N_{\sigma}^2(V^{p/2})\geq\frac{2}{p}\left[\|\nabla V^{p/2}\|_2^2+N_{\sigma}^2(V^{p/2})\right],
\end{equation}
and we deduce by combining \eqref{interm5}, \eqref{interm6} and the inequality \eqref{CKN-Young} applied to $z=V^{p/2}$ that
\begin{equation*}
	\|V^{p/2}\|_s^2\leq Cp\int_{\mathbb{R}^N}V^{p-1+1/m}(x)\,dx,
\end{equation*}
which is equivalent to \eqref{est.Moser}.
\end{proof}

We also need the following general technical result taken from \cite[Lemma~A.1]{L94}, that we state for completeness.

\begin{lemma}\label{lem.Moser2}
Let $a>1$, $b\geq0$, $c\in\mathbb{R}$, $C_0\geq1$, $C_1\geq1$ and $p_1>0$ be given such that $p_1(a-1)+c>0$. Let $\{p_j\}_{j\geq1}$ be the sequence defined by the recurrence $p_{j+1}=ap_j+c$ for $j\geq1$ and assume that $\{\gamma_j\}_{j\geq1}$ is a sequence of numbers satisfying
\begin{equation*}
	\gamma_1\leq C_1^{p_1}, \quad \gamma_{j+1}\leq C_0 p_{j+1}^b\max\{C_1^{p_{j+1}},\gamma_j^{a}\}, \quad j\geq1.
\end{equation*}
Then the sequence $\big\{\gamma_j^{1/p_j}\big\}_{j\geq1}$ is bounded.
\end{lemma}

We are now ready to start the Moser iteration scheme and thus prove that $V\in L^{\infty}(\mathbb{R}^N)$. This last property ensures that $V$ is actually positive in $\mathbb{R}^N$ and a classical solution to~Eq.~\eqref{eq.stat}.

\begin{proposition}\label{prop.br}
Let $m$ and $\sigma$ be as in~\eqref{range.exp} and $V\in\mathcal{Y}_+$ be a non-trivial solution to Eq.~\eqref{eq.stat}. Then
\begin{equation*}
	V\in L^\infty(\mathbb{R}^N) \cap C^{2,\alpha}(\mathbb{R}^N) \;\;\text{ for }\;\; \alpha\in (0,1/m)
\end{equation*}
and $V>0$ in $\mathbb{R}^N$. In particular, $V$ is a classical solution to Eq.~\eqref{eq.stat}.
\end{proposition}

\begin{proof}
Pick $p\geq2$ and $s\in(2,2^*)$. We set 
\begin{equation*}
	p_1 := \frac{m+1}{m}, \quad p_{j+1} :=  \frac{s}{2}\left( p_j+\frac{m-1}{m} \right), \quad j\ge 1,
\end{equation*}
and infer from the inequality~\eqref{est.Moser} with $p=p_j+(m-1)/m$ that, for $j\ge 1$,
\begin{equation*}
	\left(\int_{\mathbb{R}^N}V^{p_{j+1}}(x)\,dx\right)^{2/s}\leq C\left(p_j+\frac{m-1}{m}\right)\int_{\mathbb{R}^N}V^{p_j}(x)\,dx;
\end{equation*}
that is,
\begin{equation}\label{interm7}
	\|V\|_{p_{j+1}}^{p_{j+1}}\leq C\left(p_j+\frac{m-1}{m}\right)^{s/2}\|V\|_{p_j}^{sp_j/2}.
\end{equation}
Set next $\gamma_j:=\|V\|_{p_j}^{p_j}$ for $j\ge 1$. In this notation, the estimate~\eqref{interm7} becomes
\begin{equation*}
	\gamma_{j+1}\leq C\left(\frac{2}{s}p_{j+1}\right)^{s/2}\gamma_j^{s/2} \le C p_{j+1}^{s/2}\gamma_j^{s/2}, \quad j\ge 1.
\end{equation*}
Since $\|V\|_{p_1}$ is finite, we are in a position to apply Lemma~\ref{lem.Moser2} (with $a=s/2>1$, $C_0=C$ and $C_1=1+\|V\|_{p_1}$) and conclude that the sequence $\big\{\gamma_j^{1/p_j}\big\}_{j\geq1}$ is bounded. Since
\begin{equation*}
	\lim\limits_{j\to\infty}\gamma_j^{1/p_j}=\lim\limits_{j\to\infty}\|V\|_{p_j}=\|V\|_{\infty},
\end{equation*}
we have just shown that $V\in L^{\infty}(\mathbb{R}^N)$.

Owing to the just established boundedness of $V$, the function $V^{1/m}$ belongs to $L_{\mathrm{loc}}^p(\mathbb{R}^N)$ for all $p\in (1,\infty)$. Elliptic regularity applied to~\eqref{eq.stat} then implies that $V\in W_{\mathrm{loc}}^{2,p}(\mathbb{R}^N)$ for all $p\in (1,\infty)$, see \cite[Theorem~9.11]{GT01}. Sobolev embeddings guarantee that $V\in C^{1,\alpha}(\mathbb{R}^N)$ for all $\alpha\in (0,1)$, from which we deduce that $V^{1/m}\in C^{\alpha}(\mathbb{R}^N)$ for all $\alpha\in (0,1/m)$. We then infer from \cite[Theorem~9.19]{GT01} that $V\in C^{2,\alpha}(\mathbb{R}^N)$ for all $\alpha\in (0,1/m)$, which completes the proof of the regularity estimates on $V$.

Finally, assume for contradiction that there is $x_0\in\mathbb{R}^N$ such that
\begin{equation*}
	V(x_0)=0=\min_{\mathbb{R}^N} V,
\end{equation*}
the last equality being a consequence of the non-negativity of $V$. Let $R>0$. Since
\begin{equation*}
	\Delta V(x) - |x|^\sigma V(x) \le 0 \;\;\text{ and }\;\; - R^\sigma \le - |x|^\sigma \le 0, \quad x\in B(x_0,R),
\end{equation*}
it follows from the strong comparison principle that $V\equiv 0$ in $B(x_0,R)$, see \cite[Theorem~3.5]{GT01} and its consequences. As $R>0$ is arbitrary, we conclude that $V$ vanishes identically in $\mathbb{R}^N$, and a contradiction. Therefore, $V>0$ in $\mathbb{R}^N$ and the proof is complete.
\end{proof}

We next refine Proposition~\ref{prop.br} by showing that any non-trivial solution $V\in\mathcal{Y}_+$ decays algebraically as $|x|\to\infty$.
\begin{lemma}\label{lem.tail2}
Let $m$ and $\sigma$ be as in~\eqref{range.exp} and let $V\in \mathcal{Y}_+$ be a non-trivial solution to Eq.~\eqref{eq.stat}. Then, there is $C>0$ (depending only on $N$, $m$, $\sigma$ and $\|V\|_\infty$) such that
\begin{equation*}
	V(x) \le C |x|^{-\sigma m/(m-1)}, \quad x\in\mathbb{R}^N.
\end{equation*}
\end{lemma}

\begin{proof}
Let $\varepsilon\in (0,1)$ and $V_\varepsilon(x) = V(x) - \varepsilon |x|^2$ for $x\in\mathbb{R}^N$. For $a>1$ and  $R_2>R_1>1$, we define the open set
\begin{align*}
	\Omega_{a,R_1,R_2} & := \left\{ x\in B(0,R_2)\setminus \bar{B}(0,R_1)\ :\ |x|^\sigma V^{(m-1)/m}(x)>a \right\} \\
	& = \left\{ x\in B(0,R_2)\setminus \bar{B}(0,R_1)\ :\ |x|^\sigma V(x)>a V^{1/m}(x) \right\}
\end{align*}
and assume that
\begin{equation}
\begin{split}
	R_1^{\sigma+2} & \ge R_1^{\sigma+2}(a) := \frac{a}{a-1} \max\left\{ 2N, \frac{\sigma m}{m-1} \left( 1+ \frac{\sigma m}{m-1} \right) \right\}, \\
	R_2^{2} & \ge R_2^2(\varepsilon) := \frac{\|V\|_\infty}{\varepsilon}.
\end{split} \label{ph01}
\end{equation}
On the one hand, we deduce from~\eqref{eq.stat} and~\eqref{ph01} that, for $x\in \Omega_{a,R_1,R_2}$,
\begin{align*}
	-\Delta V_\varepsilon(x) + \frac{a-1}{a} |x|^\sigma V_\varepsilon(x) & = V^{1/m}(x) - \frac{|x|^\sigma}{a} V(x)  + \varepsilon \left[ 2N - \frac{a-1}{a} |x|^{\sigma+2} \right] \\
	& \le \varepsilon \left[ 2N - \frac{a-1}{a} R_1^{\sigma+2} \right] \le 0.
\end{align*}
Thus,
\begin{equation}
	-\Delta V_\varepsilon(x) + \frac{a-1}{a} |x|^\sigma V_\varepsilon(x) \le 0, \quad x\in \Omega_{a,R_1,R_2}. \label{ph02}
\end{equation}
On the other hand, setting $\theta=\sigma m/(m-1)$ and $S_A(x) := A |x|^{-\theta}$ for $x\in\mathbb{R}^N\setminus\{0\}$ with
\begin{equation}
	A\ge A_{a,R_1} := \max\left\{ a^{m/(m-1)}, \|V\|_\infty R_1^\theta \right\}, \label{ph04}
\end{equation}
we observe that, for $x\in \Omega_{a,R_1,R_2}$,
\begin{align*}
	-\Delta S_A(x) + \frac{a-1}{a} |x|^\sigma S_A(x) & = - A\theta ( 1 + \theta) |x|^{-\theta-2} + A(N-1)\theta |x|^{-\theta-2} \\
	& \qquad + \frac{A(a-1)}{a} |x|^{\sigma-\theta} \\
	& \ge A \left[ \frac{a-1}{a} |x|^{\sigma+2} - \theta(1+\theta)\right] |x|^{-\theta-2} \\
	& \ge A \left[ \frac{a-1}{a} R_1^{\sigma+2} - \theta(1+\theta)\right] |x|^{-\theta-2},
\end{align*}
which implies that, thanks to the choice~\eqref{ph01} of $R_1$,
\begin{equation}
	-\Delta S_A(x) + \frac{a-1}{a} |x|^\sigma S_A(x) \ge 0, \quad x\in \Omega_{a,R_1,R_2}. \label{ph03}
\end{equation}
Consider now $x\in\partial\Omega_{a,R_1,R_2}$. Either $|x|\in\{R_1,R_2\}$ and
\begin{align*}
	S_A(x) \ge 0 \ge V_\varepsilon(x), \quad |x|=R_2, \\
	S_A(x) = A {R_1}^{-\theta} \ge \|V\|_\infty \ge V_\varepsilon(x), \quad |x|=R_1,
\end{align*}
owing to~\eqref{ph01} and the choice~\eqref{ph04} of $A$. Or $|x|^\sigma V^{(m-1)/m}(x)=a$, which also reads $V(x)=a^{m/(m-1)} |x|^{-\theta}$ and we infer from~\eqref{ph04} that $S_A(x) \ge V(x) \ge V_\varepsilon(x)$. We have thus shown that $S_A(x)\ge V_\varepsilon(x)$ for $x\in\partial\Omega_{a,R_1,R_2}$, which we combine with~\eqref{ph02}, \eqref{ph03} and the comparison principle to conclude that
\begin{equation}
	V_\varepsilon(x) \le S_A(x) = A |x|^{-\theta}, \qquad x\in\Omega_{a,R_1,R_2}. \label{ph05}
\end{equation}
Now, fix $x\in\mathbb{R}^N$. Either $|x|^\sigma V^{(m-1)/m}(x) \le 2$, so that $V(x)\le 2^{m/(m-1)} |x|^{-\theta}$. Or $|x|^\sigma V^{(m-1)/m}(x) > 2$ and we distinguish two cases: if $|x|\le R_1(2)$, then $V(x) \le \|V\|_\infty \le \|V\|_\infty R_1^\theta(2) |x|^{-\theta}$. Otherwise, $|x|>R_1(2)$ and
\begin{equation*}
	x\in \Omega_{2,R_1(2),R_2(\varepsilon)} \;\;\text{ for any }\;\; \varepsilon\in \big( 0, \min\{1,\|V\|_\infty |x|^{-2}\}\big).
\end{equation*}
It then follows from~\eqref{ph04} and~\eqref{ph05} that
\begin{equation*}
	V(x) - \varepsilon |x|^2 \le A_{2,R_1(2)} |x|^{-\theta}, \quad \varepsilon\in \big( 0, \min\{1,\|V\|_\infty |x|^{-2}\}\big).
\end{equation*}
Letting $\varepsilon\to 0$ in the above inequality gives $V(x) \le A_{2,R_1(2)} |x|^{-\theta}$. Gathering the outcome of the above analysis completes the proof with $C= \max\left\{2^{m/(m-1)}, \|V\|_\infty R_1^\theta(2),A_{2,R_1(2)} \right\}$.
\end{proof}

\subsection{Radially symmetric stationary solutions with non-increasing profile: Existence}\label{sec.exrad}

We start with a simple general estimate that will be employed with decisive effect to prove the existence of a minimizer for $\mathcal{I}$.

\begin{lemma}\label{lem.tail}
Let $m$ and $\sigma$ be as in~\eqref{range.exp} and let $z\in\mathcal{Y}$. Then, there is $C_3>0$ (depending only on $N$, $m$ and $\sigma$) such that, for any $R>0$,
\begin{equation}\label{tail}
\begin{split}
	\|z\|_{(m+1)/m}^{(m+1)/m}&\leq C_3R^{N(m-1)/(2m)}\|z\|_{2}^{(m+1)/m}\\
	& \quad + C_3 R^{[(m+1)(\sigma_0-\sigma)]/(2m)} N_{\sigma}^{(m+1)/m}(z).
\end{split}
\end{equation}
\end{lemma}

\begin{proof}
Let $R>0$. We estimate the $L^{(m+1)/m}$-norm in the left-hand side of~\eqref{tail} by splitting the integral into a sum of integrals on $B(0,R)$ and its complement and applying suitable H\"older inequalities for each of the two terms, as follows:
\begin{equation*}
\begin{split}
\|&z\|_{(m+1)/m}^{(m+1)/m}=\int_{B(0,R)}|z|^{(m+1)/m}(x)\,dx\\
& \quad +\int_{\mathbb{R}^N\setminus B(0,R)}|x|^{\sigma(m+1)/(2m)} |z|^{(m+1)/m}(x)|x|^{-\sigma(m+1)/(2m)}\,dx\\
&\leq\left(\int_{B(0,R)} z^2(x)\,dx\right)^{(m+1)/(2m)}|B(0,R)|^{(m-1)/(2m)}\\
& \quad +\left(\int_{\mathbb{R}^N\setminus B(0,R)}|x|^{\sigma} z^2(x)\,dx\right)^{(m+1)/(2m)}\left(\int_{\mathbb{R}^N\setminus B(0,R)}|x|^{-\sigma(m+1)/(m-1)}\,dx\right)^{(m-1)/(2m)}\\
&\leq\big(\varpi_N R^N\big)^{(m-1)/(2m)} \|z\|_2^{(m+1)/(2m)}\\
& \quad + \left[ \frac{(m-1)N\varpi_N}{\sigma(m+1) - N(m-1)}\right]^{(m-1)/(2m)} \left(R^{N-\sigma(m+1)/(m-1)}\right)^{(m-1)/(2m)}N_{\sigma}^{(m+1)/(2m)}(z),
\end{split}
\end{equation*}
where $\varpi_N$ denotes the volume of the unit ball in $\mathbb{R}^N$ and we have taken into account that $\sigma>\sigma_0$, which entails the positivity of $\sigma(m+1)-N(m-1)$ and the finiteness of the integral over $\mathbb{R}^N\setminus B(0,R)$. Setting
\begin{equation*}
	C_3 := \max\left\{\varpi_N^{(m-1)/(2m)} , \left[\frac{(m-1)N\varpi_N}{\sigma(m+1)-N(m-1)}\right]^{(m-1)/2m}\right\}>0
\end{equation*}
completes the proof.
\end{proof}

We are now in a position to prove the existence of a radially symmetric minimizer to $\mathcal{I}$ with non-increasing profile.

\begin{proposition}\label{prop.exrad}
Let $m$ and $\sigma$ be as in~\eqref{range.exp}. There exists a non-trivial radially symmetric solution $V\in\mathcal{Y}_+$ to Eq.~\eqref{eq.stat} with non-increasing profile. Moreover, $V$ belongs to $L^{\infty}(\mathbb{R}^N)\cap C^{2,\alpha}(\mathbb{R}^N)$ for any $\alpha\in (0,1/m)$ and satisfies
\begin{equation}
	0< V(x) \le C |x|^{-\sigma m/(m-1)}, \quad x\in\mathbb{R}^N. \label{tc}
\end{equation}
\end{proposition}
\begin{proof}
As already mentioned in the Introduction, it is more convenient to work with the functional $\mathcal{J}$ defined by
\begin{equation*}
	\mathcal{J}(V) = \mathcal{I}(|V|^{(1-m)/m}V) =\frac{1}{2} \left( \|\nabla V\|_2^2 + N_{\sigma}^2(V) \right) - \frac{m}{m+1} \|V\|_{(m+1)/m}^{(m+1)/m}, \quad V\in\mathcal{Y}.
\end{equation*}
According to Corollary~\ref{cor.CKN} (with $s=(m+1)/m$) and Lemma~\ref{lem.lower}, the functional $\mathcal{J}$ is bounded from below and coercive on $\mathcal{Y}$, so that there is a minimizing sequence $\{V_j\}_{j\ge 1}$ in $\mathcal{Y}$ satisfying
\begin{equation}
\lim_{j\to\infty} \mathcal{J}(V_j) = \inf_{\mathcal{Y}} \mathcal{J} \;\;\text{ and }\;\; \|\nabla V_j\|_2+\|V_j\|_{(m+1)/m}+N_{\sigma}(V_j)\leq C_4, \quad j\ge 1, \label{x15}
\end{equation}
for some $C_4>0$ depending only on $N$, $m$ and $\sigma$. In particular, $\{V_j\}_{j\geq1}$ is bounded in $\dot{H}^1(\mathbb{R}^N)$ and in $L^2_{\sigma}(\mathbb{R}^N)$ and thus relatively compact in $L^2(\mathbb{R}^N)$ since $\sigma>0$. Consequently, there are $V\in \dot{H}^1(\mathbb{R}^N)\cap L^2_{\sigma}(\mathbb{R}^N)$ and a subsequence of $\{V_j\}_{j\ge 1}$ (not relabeled) such that
\begin{equation}
	\lim_{j\to\infty} \|V_j-V\|_2 = 0 \label{x16}
\end{equation}
and $\{V_j\}_{j\ge 1}$ converges weakly to $V$ in $\dot{H}^1(\mathbb{R}^N)\cap L^2_{\sigma}(\mathbb{R}^N)$. This weak convergence, along with the lower semicontinuity of the norms in $\dot{H}^1(\mathbb{R}^N)$ and $L_\sigma^2(\mathbb{R}^N)$ and~\eqref{x15}, ensures that
\begin{equation}
\|\nabla V\|_2\leq\liminf\limits_{j\to\infty}\|\nabla V_j\|_2, \quad N_\sigma(V) \le \liminf\limits_{j\to\infty} N_\sigma(V_j) \le C_4. \label{x17}
\end{equation}
Finally, we derive the strong convergence in the $L^{(m+1)/m}$-norm by employing~\eqref{tail}. Indeed, for $R>0$, it follows from~\eqref{tail}, \eqref{x15} and~\eqref{x17} that
\begin{align*}
	\|V_j-V\|_{(m+1)/m}^{(m+1)/m} & \leq C_3 R^{N(m-1)/(2m)} \|V_j-V\|_{2}^{(m+1)/m} \\
	& \quad + C_3 R^{[(m+1)(\sigma_0-\sigma)]/(2m)} N_\sigma^{(m+1)/m}(V_j-V)\\
	& \le C_3 R^{N(m-1)/(2m)} \|V_j-V\|_{2}^{(m+1)/m} \\
	& \quad + C_3 R^{[(m+1)(\sigma_0-\sigma)]/(2m)} (2C_4)^{(m+1)/m}.
\end{align*}
Thanks to the convergence~\eqref{x16}, we may take the limit $j\to\infty$ in the above inequality to obtain
\begin{equation*}
	\limsup\limits_{j\to\infty}\|V_j-V\|_{(m+1)/m}^{(m+1)/m} \leq C_3 (2C_4)^{(m+1)/m} R^{[(m+1)(\sigma_0-\sigma)]/(2m)}.
\end{equation*}
The above inequality being valid for any $R>0$, we take the limit $R\to\infty$ and, taking into account the range $\sigma>\sigma_0$, we conclude that $\{V_j\}_{j\ge 1}$ converges to $V$ in $L^{(m+1)/m}(\mathbb{R}^N)$. Together with~\eqref{x15} and~\eqref{x17}, this convergence implies that
\begin{equation*}
	\inf_{\mathcal{Y}} \mathcal{J} \le \mathcal{J}(V) \le \liminf_{j\to\infty} \mathcal{J}(V_j) = \inf_{\mathcal{Y}} \mathcal{J},
\end{equation*}
so that $V$ is a minimizer for $\mathcal{J}$ on $\mathcal{Y}$. In particular, $V$ is a weak solution to its associated Euler-Lagrange equation, which is Eq.~\eqref{eq.stat}. Moreover, the properties of the symmetric decreasing rearrangement and the P\'olya-Szeg\"o inequality, see \cite[Section~3.3 \& Lemma~7.17]{LL01} for instance, directly imply that
\begin{equation*}
\mathcal{J}(V)=\mathcal{J}(|V|)\geq\mathcal{J}(|V|^*),
\end{equation*}
which ensures that there exists a radially symmetric minimizer $V\in\mathcal{Y}_+$ to $\mathcal{J}$ on $\mathcal{Y}$ with non-increasing profile. We finally observe that, for any $W\in\mathcal{Y}_+$, $W\not\equiv 0$, and $\lambda>0$,
\begin{equation*}
	\frac{\mathcal{J}(\lambda W)}{\lambda^2} = \frac{1}{2} \left( \|\nabla W\|_2^2 + N_{\sigma}^2(W) \right) - \frac{m}{m+1} \lambda^{-(m-1)/m} \|W\|_{(m+1)/m}^{(m+1)/m},
\end{equation*}
which is negative for $\lambda$ small enough (depending on $W$). Consequently, $\mathcal{J}$ takes negative values on $\mathcal{Y}$ and thus $\mathcal{J}(V)<0$, which implies that $V\not\equiv 0$. Finally, since $V\in\mathcal{Y}_+$ is a non-trivial solution to~\eqref{eq.stat}, we infer from Proposition~\ref{prop.br} and Lemma~\ref{lem.tail2} that $V$ satisfies the properties listed in Proposition~\ref{prop.exrad}.
\end{proof}

\subsection{Radially symmetric stationary solutions with non-increasing profile: Uniqueness}\label{sec.uniqrad}

Let $m$ and $\sigma$ be as in~\eqref{range.exp} and let $V\in\mathcal{Y}_+\cap C^2(\mathbb{R}^N)$ be a radially symmetric solution to Eq.~\eqref{eq.stat} with non-increasing profile. Then $V$ is a solution to the differential equation
\begin{equation}\label{rad.ODE}
	V''(r)+\frac{N-1}{r}V'(r)-r^{\sigma}V(r)+V^{1/m}(r)=0, \quad r=|x|\in (0,\infty),
\end{equation}
where we use again the notation $V(r)=V(|x|)=V(x)$, $x\in\mathbb{R}^N$ for simplicity. Inserting formally the ansatz $V(r)\sim Ar^{-\theta}$ as $r\to\infty$ into \eqref{rad.ODE}, we find by direct calculation that
\begin{equation*}
A\theta(\theta+2-N)r^{-\theta-2}-Ar^{-\theta+\sigma}+A^{1/m}r^{-\theta/m}\sim0, \quad r\to\infty,
\end{equation*}
and, since $\sigma>\sigma_0>0$, the only way to balance dominating terms is to let $-\theta/m=-\theta+\sigma$; that is, $\theta=m\sigma/(m-1)$, and $A=1$. The expected behavior of $V$ at infinity is then given by
\begin{equation*}
	\lim\limits_{r\to\infty}r^{m\sigma/(m-1)}V(r)=1,
\end{equation*}
which is equivalent to~\eqref{beh.stat} (recall the relationship $\varphi=V^{1/m}$) and matches the tail estimate derived in Proposition~\ref{prop.exrad}. We now give a rigorous proof of this expected asymptotic behavior.

\begin{lemma}\label{lem.beh.inf}
Let $m$ and $\sigma$ be as in~\eqref{range.exp} and consider a radially symmetric solution $V\in\mathcal{Y}_+\cap C^2(\mathbb{R}^N)$ to Eq.~\eqref{eq.stat} with non-increasing profile. Then
\begin{equation}\label{beh.inf}
	\lim\limits_{r\to\infty}r^{m\sigma/(m-1)}V(r)=1.
\end{equation}
\end{lemma}
\begin{proof}
We introduce
\begin{equation}\label{def.W}
	W(r):=V(r)r^{m\sigma/(m-1)}, \quad r\ge 0,
\end{equation}
and find by direct calculation that
\begin{equation*}
V'(r)=\left(rW'(r)-\frac{m\sigma}{m-1}W(r)\right)r^{-m\sigma/(m-1)-1}
\end{equation*}
and
\begin{equation*}
V''(r)=\left[r^2W''(r)-\frac{2m\sigma}{m-1}rW'(r)+\frac{m\sigma}{m-1}\left(1+\frac{m\sigma}{m-1}\right)W(r)\right]r^{-m\sigma/(m-1)-2}.
\end{equation*}
We thus find that $W$ solves the following differential equation:
\begin{equation}\label{W.ODE}
\begin{split}
r^2W''(r)&+\left(N-1-\frac{2m\sigma}{m-1}\right)rW'(r)+\frac{m\sigma}{m-1}\left(2-N+\frac{m\sigma}{m-1}\right)W(r)\\
&-r^{\sigma+2}W(r)+r^{\sigma+2}W^{1/m}(r)=0, \quad r\in (0,\infty).
\end{split}
\end{equation}

\medskip

\noindent \textbf{Step~1. Monotonicity for $r$ large.} In a first step, we prove that there is $R_0>0$ sufficiently large such that, either $W'>0$ on $(R_0,\infty)$, or $W'<0$ on $(R_0,\infty)$. In order to prove this claim, pick $R_0$ large enough (to be chosen later) and assume that $W$ has a point of maximum at $r_1>R_0$. Then $r_1W'(r_1)=0$ and $r_1^2W''(r_1)<0$ and it follows from~\eqref{W.ODE} that
\begin{equation*}
\begin{split}
	0&\leq\left(\frac{m\sigma}{m-1}\left(2-N+\frac{m\sigma}{m-1}\right)-r_1^{\sigma+2}\right)W(r_1)+r_1^{\sigma+2}W^{1/m}(r_1)\\
	&=r_0^{\sigma+2}W^{1/m}(r_1)\left[1-\left(1-\frac{C_5}{r_1^{\sigma+2}}\right)W^{(m-1)/m}(r_1)\right],
\end{split}
\end{equation*}
where
\begin{equation*}
	C_5 := \frac{m\sigma}{m-1}\left[2-N+\frac{m\sigma}{m-1}\right].
\end{equation*}
The previous estimate implies that, if $r_1>R_0>C_5^{1/(\sigma+2)}$ is a point of maximum for $W$, then
\begin{equation}\label{interm.max}
	W^{(m-1)/m}(r_1) \leq\frac{r_1^{\sigma+2}}{r_1^{\sigma+2}-C_5}.
\end{equation}
Similarly, if $W$ has a point of minimum at $r_0>R_0>C_5^{1/(\sigma+2)}$, then
\begin{equation}\label{interm.min}
	W^{(m-1)/m}(r_0) \geq \frac{r_0^{\sigma+2}}{r_0^{\sigma+2}-C_5}.
\end{equation}
Assume now for contradiction that $W$ has a point of minimum at $r_0>R_0>C_5^{1/(\sigma+2)}$ and then a point of maximum at $r_1>r_0$. Combining~\eqref{interm.max} and~\eqref{interm.min} gives
\begin{equation*}
W(r_1)^{(m-1)/m}\leq \frac{r_1^{\sigma+2}}{r_1^{\sigma+2}-C_5} \leq\frac{r_0^{\sigma+2}}{r_0^{\sigma+2}-C_5}\leq W(r_0)^{(m-1)/m},
\end{equation*}
which implies that $W$ is constant on $(r_0,r_1)$ and thus (by standard theory of ordinary differential equations) $W$ is constant for $r\in(0,\infty)$. This implies that, either $W\equiv1$, which is a contradiction since $W(0)=0$ by construction, or $W\equiv 0$, which contradicts $V\not\equiv 0$. Therefore, $W(r)$ cannot have an oscillatory behavior as $r\to\infty$ and we have thus established that there is $R_0>0$ such that, either $W'>0$ on $(R_0,\infty)$, or $W'<0$ on $(R_0,\infty)$.

\medskip

\noindent \textbf{Step 2. Behavior at infinity.} We infer from the previous step that there exists
\begin{equation*}
	L:=\lim\limits_{r\to\infty}W(r)\in [0,\infty].
\end{equation*}
The goal of this step is to prove that $L=1$. To this end, we want to rule out first the possibility that either $L=0$ or $L=\infty$. The latter is actually discarded by Proposition~\ref{prop.exrad}, so that we turn to the proof of the former by a contradiction argument, as indicated below.

Assume for contradiction that $L=0$. This fact, together with Step~1, entails that $W'<0$ on $(R_0,\infty)$ for some $R_0>0$. We deduce from Eq.~\eqref{W.ODE} that
\begin{equation*}
\begin{split}
	r^2W''(r)&=r^{\sigma+2}W^{1/m}(r) \left[ W^{(m-1)/m}(r) - 1 - C_5 r^{-(\sigma+2)} W^{(m-1)/m}(r) \right]\\
	& \quad + \left(\frac{2m\sigma}{m-1}-N+1\right)rW'(r).
\end{split}
\end{equation*}
Since $\sigma>\sigma_0$, we find that
\begin{equation*}
	\frac{2m\sigma}{m-1}-N+1>\frac{2mN(m-1)}{(m-1)(m+1)}-N+1=\frac{N(m-1)}{m+1}+1>0,
\end{equation*}
and it readily follows from the previous equality and the assumption $L=0$ that there exist $\delta>0$ and $R_1>R_0>0$ such that
\begin{equation*}
	r^2W''(r)\leq-\delta, \quad r\in(R_1,\infty).
\end{equation*}
This is equivalent to
\begin{equation}\label{interm8}
	W''(r)\leq-\frac{\delta}{r^2}, \quad r\in(R_1,\infty),
\end{equation}
and, by integrating twice~\eqref{interm8} over $(R_1,\infty)$, we easily deduce that
\begin{equation*}
	W(r) \leq W(R_1) + \left( W'(R_1) -\frac{\delta}{R_1} \right)(r-R_1) + \delta\ln\left(\frac{r}{R_1}\right), \quad r>R_1.
\end{equation*}
Recalling that $W'(R_1)\le 0$, we readily infer from the above inequality that $W(r)<0$ for $r>R_1$ sufficiently large, in contradiction to the positivity of $W$.


We have thus shown that $L\in(0,\infty)$. We now argue exactly as in \cite[Lemma~2.10]{IL13} and, since $|W'|\in L^1((R_0,\infty))$, we infer that there exists a sequence $\{r_k\}_{k\geq1}$ such that $r_k\to\infty$ and $r_kW'(r_k)\to 0$ as $k\to\infty$. Applying \cite[Lemma~2.9]{IL13} to the function $r\mapsto rW'(r)$ if $W$ is non-decreasing on $(R_0,\infty)$ or to $r\mapsto -rW'(r)$ otherwise, we deduce that there is a sequence $\{\rho_k\}_{k\geq1}$ such that $\rho_k\to\infty$, $\rho_k W'(\rho_k)\to 0$ and $\rho_k(\rho_k W''(\rho_k)+W'(\rho_k))\to 0$ as $k\to\infty$. The latter limit also ensures that $\rho_k^2W''(\rho_k)\to 0$ as $k\to\infty$. We next obtain by evaluating \eqref{W.ODE} at $r=\rho_k$ and passing to the limit as $k\to\infty$ that $L^{1/m}=L$ and, since $L\in(0,\infty)$, it follows that $L=1$. The claimed behavior~\eqref{beh.inf} follows then from~\eqref{def.W}.
\end{proof}
We are now in a position to complete the proof of Theorem~\ref{th.uniq} within the class of radially symmetric functions with non-increasing profile.
	
\begin{proposition}\label{prop.exuniq}
Let $m$ and $\sigma$ be as in~\eqref{range.exp}. There exists a unique non-trivial radially symmetric solution $V_*\in\mathcal{Y}_+$ to Eq.~\eqref{eq.stat} with non-increasing profile. Moreover, $V_*$ is positive in $\mathbb{R}^N$,  belongs to $L^{\infty}(\mathbb{R}^N)\cap C^{2,\alpha}(\mathbb{R}^N)$ for any $\alpha\in (0,1/m)$ and satisfies~\eqref{beh.inf}.
\end{proposition}
\begin{proof}
The existence and regularity statements in Proposition~\ref{prop.exuniq} follow from Proposition~\ref{prop.exrad}, while the precise behavior~\eqref{beh.inf} for large values of $|x|$ is provided by Lemma~\ref{lem.beh.inf}. We are thus left with proving the uniqueness statement. To this end, we employ an optimal barrier argument stemming from \cite{YeYin}, which is completely similar to the one used in \cite{ILS24, IM26}. We give a sketch of the proof for the sake of completeness. Assume for contradiction that there exist two solutions $V_1$ and $V_2$ to~\eqref{eq.stat} with the properties described in Proposition~\ref{prop.exuniq}. Owing to their radial symmetry, $V_1$ and $V_2$ are actually two solutions to~\eqref{rad.ODE}. We may also assume that $V_2(0)<V_1(0)$ (otherwise, since $V_1'(0)=V_2'(0)=0$ by radial symmetry, we have $V_1\equiv V_2$). For $\lambda>0$, we set
\begin{equation*}
	W_{\lambda}(r)=\lambda^{-2m/(m-1)}V_2(\lambda r), \quad r\in[0,\infty),
\end{equation*}
and deduce by direct calculation that $W_{\lambda}$ is a solution to
\begin{equation}\label{eq.resc}
	W_{\lambda}''(r)+\frac{N-1}{r}W_{\lambda}'(r)+W_{\lambda}^{1/m}(r)-\lambda^{\sigma+2} r^\sigma W_{\lambda}(r)=0, \quad r\in (0,\infty).
\end{equation}
Moreover, we infer from~\eqref{beh.inf} (which is satisfied by $V_2$) that
\begin{equation}\label{interm10}
\begin{split}
	\lim\limits_{r\to\infty}r^{m\sigma/(m-1)}W_{\lambda}(r)&=\lim\limits_{r\to\infty}(\lambda r)^{m\sigma/(m-1)}\lambda^{-2m/(m-1)-m\sigma/(m-1)}V_2(\lambda r)\\
	&=\lambda^{-m(\sigma+2)/(m-1)}.
\end{split}
\end{equation}
Picking $\lambda\in(0,1)$, we obtain from~\eqref{beh.inf} (which is satisfied by  $V_1$) and~\eqref{interm10} that there is $r_{\lambda}>0$ such that
\begin{equation}\label{interm11}
	W_{\lambda}(r)>V_1(r), \quad r\in(r_{\lambda},\infty),
\end{equation}
and we take the smallest value of $r_{\lambda}$ for which~\eqref{interm11} holds true. Since, for $\lambda'<\lambda$, the monotonicity of $V_2$ implies that
\begin{equation*}
	(\lambda')^{-2m/(m-1)}>\lambda^{-2m/(m-1)}, \quad V_2(\lambda'r)\geq V_2(\lambda r),
\end{equation*}
we infer that $W_{\lambda'}(r)>W_{\lambda}(r)$ for any $r\in[0,\infty)$ and thus $r_{\lambda'}<r_{\lambda}$. Moreover, 
\begin{equation*}
	\lim\limits_{\lambda\to0}\inf\limits_{[0,r_{1/2}]} W_{\lambda} = \infty,
\end{equation*}
so that there is $\lambda_*\in (0,1/2)$ such that
\begin{equation*}
	\inf_{[0,r_{1/2}]} W_{\lambda} > \max_{[0,r_{1/2}]} V_1, \quad \lambda\in (0,\lambda_*).
\end{equation*}
Since
\begin{equation*}
	W_{\lambda}(r) > W_{\lambda_*}(r) > W_{1/2}(r) > V_1(r), \quad (r,\lambda)\in (r_{1/2},\infty)\times (0,\lambda_*),
\end{equation*}
we conclude that $W_{\lambda}>V_1$ in $[0,\infty)$ for $\lambda>0$ small enough. We can thus introduce the optimal parameter
\begin{equation*}
	\lambda_0:=\sup\{\lambda\in(0,1): W_{\lambda}>V_1 \ \text{ in } \ [0,\infty)\},
\end{equation*}
and observe that $\lambda_0\in(0,1)$, since $V_2(0)<V_1(0)$. Assume for contradiction that $r_{\lambda_0}>0$. It then follows by the definitions of $\lambda_0$ and $r_{\lambda_0}$ that
\begin{equation*}
	W_{\lambda_0}(r)\geq V_1(r), \quad r\in[0,\infty), \quad W_{\lambda_0}(r_{\lambda_0})=V_1(r_{\lambda_0}),
\end{equation*}
which also means that $r_{\lambda_0}$ is a point of minimum for $W_{\lambda_0}-V_1$ and thus
\begin{equation*}
	W_{\lambda_0}'(r_{\lambda_0})=V_1'(r_{\lambda_0}), \quad W_{\lambda_0}''(r_{\lambda_0})\geq V_1''(r_{\lambda_0}).
\end{equation*}
Evaluating~\eqref{eq.resc} (with the solution $W_{\lambda_0}$) and~\eqref{rad.ODE} (with the solution $V_1$) at $r=r_{\lambda_0}$, we deduce that
\begin{equation*}
	0\geq(1-\lambda_0^{\sigma+2})r_{\lambda_0}^\sigma V_1(r_{\lambda_0}),
\end{equation*}
which contradicts the strict positivity of $V_1$. Consequently, $r_{\lambda_0}=0$ and the optimality of $\lambda_0$ ensures that
\begin{equation*}
W_{\lambda_0}(0)=\lambda_0^{-2m/(m-1)}V_2(0)=V_1(0).
\end{equation*}
Arguing as in \cite[Lemma~4.12]{ILS24} and using the expansions of $V_1$ and $V_2$ as $r\to 0$ (which are the same as the ones derived in \cite[Lemmas~4.2 and~4.3]{ILS24}), we discard this possibility and conclude the proof of the uniqueness.
\end{proof}

\section{Convergence}\label{sec.cv}

This section is dedicated to the proof of Theorem~\ref{th.conv} and to the last part of Theorem~\ref{th.uniq}, establishing uniqueness of the non-trivial stationary solution without the assumption of radial symmetry with non-increasing profile. To this end, given an initial condition $u_0\in L^\infty_+(\mathbb{R}^N)\cap \mathcal{X}$, let $v$ be the corresponding solution to~\eqref{CP2} and denote its $\omega$-limit set by $\omega(u_0)$, which is defined as follows: we say that $V\in\omega(u_0)$ if there is a sequence $\{s_j\}_{j\geq1}$ such that
\begin{equation}\label{omega.limit}
\lim\limits_{j\to\infty}s_j=\infty, \quad \lim\limits_{j\to\infty}\|v^m(s_j)-V\|_{2}=0.
\end{equation}
Our goal is thus to characterize the elements of $\omega(u_0)$ for suitable initial conditions $v_0\in\mathcal{X}$ and is achieved by the following result.

\begin{proposition}\label{prop.omega}
Let $m$ and $\sigma$ be as in \eqref{range.exp} and $u_0\in L^\infty_+(\mathbb{R}^N)\cap \mathcal{X}$. The set $\omega(u_0)$ is non-empty and any element $V\in\omega(u_0)$ is a non-trivial weak solution to the elliptic equation~\eqref{eq.stat}, which satisfies
\begin{equation}
	V^{1/m}(x) \ge \left( \frac{e^{(m-1)s} - 1}{e^{(m-1)s}} \right)^{1/(m-1)} v(s,x), \quad (s,x)\in (0,\infty)\times \mathbb{R}^N. \label{fg01}
\end{equation}
Moreover, if $u_0$ is radially symmetric with non-increasing profile, then any element $V\in\omega(u_0)$ is radially symmetric with non-increasing profile as well.
\end{proposition}

\begin{proof}
We deduce from Proposition~\ref{prop.Lyap} and~\eqref{ubv} that $v(s)\in L^\infty_+(\mathbb{R}^N)\cap \mathcal{X}$ for any $s\geq0$ and that $\{v^m(s)\}_{s\geq0}$ is uniformly bounded in $\dot{H^1}(\mathbb{R}^N)\cap L^{2}_{\sigma}(\mathbb{R}^N)\cap L^{(m+1)/m}(\mathbb{R}^N)$. On the one hand, since $(m+1)/m<2$, it follows that $\dot{H^1}(\mathbb{R}^N)\cap L^{(m+1)/m}(\mathbb{R}^N)$ embeds continuously in $H^1(\mathbb{R}^N)$. On the other hand, since $\sigma>\sigma_0>0$, $H^1(\mathbb{R}^N)\cap L^2_{\sigma}(\mathbb{R}^N)$ is compactly embedded in $L^2(\mathbb{R}^N)$. Putting together these embeddings, we infer that the set $\{v^m(s)\}_{s\geq0}$ is relatively compact in $L^2(\mathbb{R}^N)$, from which we readily deduce that $\omega(u_0)$ is non-empty.

Pick $V\in\omega(u_0)$ and let $\{s_j\}_{j\geq1}$ be a sequence such that~\eqref{omega.limit} holds true. To show that $V$ solves~\eqref{eq.stat}, we proceed along the lines of \cite[Theorem~1.1]{LP85} and first observe that Lemma~\ref{lem.lower} and~\eqref{diss} imply that
\begin{equation*}
\begin{split}
	\int_0^{s}\Big\|\partial_s\big(v^{(m+1)/2}\big)(s_*)\Big\|_2^2\,ds_* & \leq \frac{(m+1)^2}{4m} \left[\mathcal{I}(u_0)-\mathcal{I}(v(s))\right] \\
	& \le \frac{(m+1)^2}{4m} \left[\mathcal{I}(u_0) + \mathcal{I}_0 \right]
\end{split}
\end{equation*}
for any $s\geq 0$, from which we readily deduce that 
\begin{equation}\label{interm13}
	K_1 := \int_0^{\infty}\Big\|\partial_s\big(v^{(m+1)/2}\big)(s)\Big\|_2^2\,ds < \infty.
\end{equation}
Set now
\begin{equation*}
v_j(s):=v(s+s_j), \quad s\in[0,1], \quad j\geq1.
\end{equation*}
On the one hand, the following numerical inequality, valid for any $(X,Y)\in[0,\infty)^2$ since $m>(m+1)/2$,
\begin{align*}
	\big|X^m - Y^m\big| & \le \frac{2m}{m+1} \max\left\{ X^{(m+1)/2} , Y^{(m+1)/2} \right\}^{(m-1)/(m+1)} \left| X^{(m+1)/2} - Y^{(m+1)/2} \right| \\
	& = \frac{2m}{m+1} \max\left\{X,Y \right\}^{(m-1)/2} \left| X^{(m+1)/2} - Y^{(m+1)/2} \right|
\end{align*}
implies that, for any $s\in[0,1]$, $j\geq1$ and $x\in\mathbb{R}^N$,
\begin{align*}
	\big|v_j^m(s,x) - v_j^m(0,x)\big|&\leq\frac{2m}{m+1} \max\left\{ \|v_j(s)\|_\infty,\|v_j(0)\|_\infty \right\}^{(m-1)/2}\\
	&\quad \times\big|v_j^{(m+1)/2}(s,x) - v_j^{(m+1)/2}(0,x)\big| \\
	&\leq\frac{2m}{m+1} \max\left\{ \|v(s_j+s)\|_\infty,\|v(s_j)\|_\infty \right\}^{(m-1)/2}\\
	&\quad \times\int_{s_j}^{s_j+s} \left| \partial_s \left( v^{(m+1)/2} \right)(s_*,x) \right|\,ds_*.
\end{align*}
The uniform boundedness~\eqref{ubv} of $v$ in $L^{\infty}((\ln{m}/(m-1),\infty)\times\mathbb{R}^N)$, together with the previous estimate and an application of H\"older's inequality, further gives
\begin{equation*}
	\sup_{s\in [0,1]} \big\| v_j^m(s) - v^m(s_j) \big\|_2 \le K_2\left( \int_{s_j}^{\infty} \left\| \partial_s \left( v^{(m+1)/2} \right)(s_*,x) \right\|_2^2\ ds_* \right)^{1/2},
\end{equation*}
with $K_2 := 2C_2(u_0)^{(m-1)/2}$, where we recall that the constant $C_2(u_0)$ is introduced in~\eqref{ubv}. Therefore, by~\eqref{omega.limit} and~\eqref{interm13},
\begin{equation}\label{conv1}
	\lim_{j\to\infty}\sup_{s\in [0,1]}\big\|v_j^m(s)-V\big\|_2 = 0.
\end{equation}
Moreover, combining~\eqref{conv1} and the numerical inequality
\begin{equation}\label{num.ineq}
	|X-Y|^m \le|X^m-Y^m|, \quad (X,Y)\in[0,\infty)^2,
\end{equation}
which is valid due to $m>1$, gives
\begin{equation}\label{conv3}
	\lim_{j\to\infty} \sup_{s\in [0,1]} \big\|v_j(s) - V^{1/m}\big\|_{2m} = 0.
\end{equation}
On the other hand, the uniform boundedness~\eqref{ubv} of $v^m$ in $\dot{H}^1(\mathbb{R}^N)\cap L^2_{\sigma}(\mathbb{R}^N)$ ensures that $\big\{\nabla v_j^m\big\}_{j\geq1}$ and $\{v_j^m\}_{j\geq1}$ are uniformly bounded in $L^{\infty}((0,1),L^2(\mathbb{R}^N))$ and in $L^{\infty}((0,1),L^2_{\sigma}(\mathbb{R}^N))$, respectively. We infer from these bounds and~\eqref{conv1} that there exists a subsequence of $\{v_j\}_{j\geq1}$ (not relabeled for simplicity) such that
\begin{equation}\label{conv2}
\begin{split}
	\nabla v^m(s_j) \rightharpoonup \nabla V & \;\;\text{ in }\;\; L^2(\mathbb{R}^N), \\
	\nabla v^m_j \stackrel{*}{\rightharpoonup} \nabla V & \;\;\text{ in }\;\; L^\infty((0,1),L^2(\mathbb{R}^N)), \\
	v^m_j \stackrel{*}{\rightharpoonup} V & \;\;\text{ in }\;\; L^\infty((0,1),L_\sigma^2(\mathbb{R}^N)).
\end{split}
\end{equation}

We finally write the weak formulation of Eq.~\eqref{eq2} on $(s_j,s_j+1)$ with $\vartheta\in C_c^{\infty}(\mathbb{R}^N)$ as test function, which yields
\begin{align*}
 	\int_{\mathbb{R}^N} \big[v_j(1,x) - v_j(0,x)] \vartheta(x)\ dx & + \int_0^1 \int_{\mathbb{R}^N} \nabla v_j^m(s,x)\cdot \nabla\vartheta\ dxds \\
 	& \quad = \int_0^1 \int_{\mathbb{R}^N} \left[v_j(s,x)-|x|^\sigma v_j^m(s,x)\right] \vartheta(x)\ dxds.
 \end{align*}
We can thus pass to the limit as $j\to\infty$ in the previous identity and deduce from \eqref{conv1}, \eqref{conv2} and \eqref{conv3} that
\begin{align*}
\int_{\mathbb{R}^N}\nabla V(x)\cdot\nabla\vartheta(x)\,dxds =\int_{\mathbb{R}^N}\left[V^{1/m}(x)- |x|^\sigma V(x)\right]\vartheta(x)\,dxds,
\end{align*}
whence $V$ is a weak solution to~\eqref{eq.stat}. In order to check that $V\not\equiv 0$, we fix $s_0>0$ and recall the time monotonicity~\eqref{x11}
\begin{equation*}
	\left( \frac{e^{(m-1)s_j} - 1}{e^{(m-1)s_j}} \right)^{1/(m-1)} v(s_j,x) \ge \left( \frac{e^{(m-1)s_0} - 1}{e^{(m-1)s_0}} \right)^{1/(m-1)} v(s_0,x), \quad x\in\mathbb{R}^N,
\end{equation*}
which is valid for any $j\ge 1$ such that $s_j>s_0$. We then let $j\to\infty$ in the above inequality and deduce from the convergence~\eqref{conv3} that
\begin{equation*}
	V^{1/m}(x) \ge \left( \frac{e^{(m-1)s_0} - 1}{e^{(m-1)s_0}} \right)^{1/(m-1)} v(s_0,x), \quad x\in\mathbb{R}^N,
\end{equation*}
which readily implies~\eqref{fg01}, as well as $V\not\equiv 0$, since $v(s_0)\not\equiv 0$ for all $s_0>0$ by Proposition~\ref{prop.v}.

Finally, if $u_0$ is radially symmetric with non-increasing profile, then it follows from Lemma~\ref{lem.rad} that $v(s)$ shares the same properties for $s\geq0$ and thus $V$ is radially symmetric with non-increasing profile, completing the proof.
\end{proof}

We are only left to finish the proofs of Theorems~\ref{th.uniq} and~\ref{th.conv}, which strongly rely on Proposition~\ref{prop.omega} and the detailed study of the elliptic equation~\eqref{eq.stat} performed in Section~\ref{sec.stso}.

\begin{proof}[Proof of Theorems~\ref{th.uniq} and~\ref{th.conv}]
Several steps are needed to be able to handle arbitrary initial conditions $u_0$ in $L^\infty_+(\mathbb{R}^N)\cap \mathcal{X}$.

\medskip

\noindent\textbf{Step~1.} Assume first that $u_0\in L^\infty_+(\mathbb{R}^N)\cap \mathcal{X}$ is a radially symmetric initial condition with non-increasing profile. It follows from Proposition~\ref{prop.omega} that any element $V\in\omega(u_0)$ is a radially symmetric function with non-increasing profile in $L^\infty_+(\mathbb{R}^N)\cap \mathcal{X}$, as well as a non-trivial weak solution to~\eqref{eq.stat}. Proposition~\ref{prop.exuniq} then entails that $\omega(u_0)=\{V_*\}$ is a singleton and a classical compactness argument leads us to the convergence
\begin{equation}
	\lim_{s\to\infty} \|v^m(s) - V_*\|_2 = 0. \label{conv4}
\end{equation}

\medskip

\noindent\textbf{Step~2.} Assume now that $u_0\in L^\infty_+(\mathbb{R}^N)\cap \mathcal{X}$ is such that there are two radially symmetric initial conditions with non-increasing profile $u_{0,b}\in L^\infty_+(\mathbb{R}^N)\cap \mathcal{X}$ and $u_{0,a}\in L^\infty_+(\mathbb{R}^N)\cap \mathcal{X}$ satisfying $u_{0,b} \le u_0 \le u_{0,a}$ in $\mathbb{R}^N$. Denoting the corresponding solutions to~\eqref{CP2} by $v_{b}$ and $v_{a}$, respectively, the comparison principle recalled in Proposition~\ref{prop.wp} implies that
\begin{equation*}
	v_{b}(s,x) \le v(s,x) \le v_{a}(s,x), \quad (s,x)\in (0,\infty)\times\mathbb{R}^N,
\end{equation*}
so that
\begin{align*}
	\|v^m(s)-V_*\|_2^2 & = \big\|\big(v^m(s)-V_*\big)_+\big\|_2^2 + \big\|\big(V_*-v^m(s)\big)_+\big\|_2^2 \\
	& \le \big\|\big(v_{a}^m(s)-V_*\big)_+\big\|_2^2 + \big\|\big(V_*-v_{b}^m(s)\big)_+\big\|_2^2 \\
	& \le \|v_{a}^m(s)-V_*\|_2^2 + \|v_{b}^m(s)-V_*\|_2^2.
\end{align*}
Since $v_{b}$ and $v_{a}$ satisfy~\eqref{conv4} by \textbf{Step~1}, an immediate consequence of the above inequality is that $\{v^m(s)\}_{s>0}$ converges to $V_*$ in $L^2(\mathbb{R})$; that is, $\omega(u_0)=\{V_*\}$.

\medskip

\noindent\textbf{Step~3.} We now complete the proof of Theorem~\ref{th.uniq}. Let $V\in\mathcal{Y}_+$ be a non-trivial stationary solution to~\eqref{eq.stat}. Then $V\in L^\infty(\mathbb{R}^N)$ by Proposition~\ref{prop.br}, so that $V^{1/m}$ is a non-trivial function in $L^\infty_+(\mathbb{R}^N)\cap \mathcal{X}$ and a stationary solution to~\eqref{eq2}. Moreover, the positivity of $V$ established in Proposition~\ref{prop.br} guarantees that there is a radially symmetric initial condition with non-increasing profile $u_{0,b}\in L^\infty_+(\mathbb{R}^N)\cap \mathcal{X}$ such that $u_{0,b}\le V^{1/m}$ in $\mathbb{R}^N$. Also, owing to Proposition~\ref{prop.br} and Lemma~\ref{lem.tail2}, there is $C>0$ depending on $N$, $m$, $\sigma$ and $\|V\|_\infty$ such that $V^{1/m}(x) \le C u_{0,a}(x)$ with $u_{0,a}(x) := (1+|x|)^{-\sigma/(m-1)}$ for $x\in\mathbb{R}^N$. Taking into account that $\sigma>\sigma_0$, we realize that $u_{0,a}$ belongs to $L^\infty_+(\mathbb{R}^N)\cap \mathcal{X}$ and is non-negative and radially symmetric with non-increasing profile. We are then in a position to apply \textbf{Step~2} to conclude that $V=V_*$, thereby completing the proof of Theorem~\ref{th.uniq} with $\varphi=V_*^{1/m}$.

\medskip

\noindent\textbf{Step~4.} Consider now an arbitrary initial condition $u_0\in L^\infty_+(\mathbb{R}^N)\cap \mathcal{X}$ and let $v$ be the corresponding solution to~\eqref{CP2}. A direct consequence of Proposition~\ref{prop.omega} and \textbf{Step~3} is that $\omega(u_0)=\{V_*\}$ and the same classical compactness argument used in \textbf{Step~1} then implies that~\eqref{conv4} holds true for $v$ as well. Furthermore, by~\eqref{fg01},
\begin{equation}
	v(s,x) \le \left( \frac{e^{(m-1)s}}{e^{(m-1)s} - 1} \right)^{1/(m-1)} V_*^{1/m}(x), \quad (s,x)\in (0,\infty)\times \mathbb{R}^N. \label{fg03}
\end{equation}
We then infer from~\eqref{num.ineq}, \eqref{fg03} and H\"older's inequality that, for $s>0$ and $R\ge 1$,
\begin{align*}
	& \big\|v(s)-V_*^{1/m}\big\|_{m+1}^{m+1} \le \|v^m(s)-V_*\|_{(m+1)/m}^{(m+1)/m} \\
	& \quad \le \int_{B(0,R)} |v^m(s)-V_*|^{(m+1)/m}\ dx \\
	& \quad\quad + \int_{\mathbb{R}^N\setminus B(0,R)} |v^m(s)+V_*|^{(m+1)/m}\ dx \\
	& \quad\le  |B(0,R)|^{(m-1)/(2m)} \left( \int_{B(0,R)} |v^m(s)-V_*|^{2}\ dx \right)^{(m+1)/(2m)} \\
	& \quad\quad + 2^{1/m} \int_{\mathbb{R}^N\setminus B(0,R)} \left( v^{m+1}(s) + V_*^{(m+1)/m}\right)\ dx \\
	& \quad\le \big(\varpi_N R^N \big)^{(m-1)/(2m)} \big\|v^m(s)-V_*\big\|_2^{(m+1)/m} \\
	& \quad\quad + 2^{1/m} \left( 1 + \left[ \frac{e^{(m-1)s}}{e^{(m-1)s}-1} \right]^{(m+1)/(m-1)} \right)\int_{\mathbb{R}^N\setminus B(0,R)} V_*^{(m+1)/m}\ dx.
\end{align*}
According to~\eqref{conv4}, we may take the limit $s\to\infty$ in the above inequality and find
\begin{equation*}
	\limsup_{s\to\infty} \big\|v(s)-V_*^{1/m}\big\|_{m+1}^{m+1} \le 2^{(m+1)/m} \int_{\mathbb{R}^N\setminus B(0,R)} V_*^{(m+1)/m}\ dx.
\end{equation*}
Recalling that $V_*\in L^{(m+1)/m}(\mathbb{R}^N)$, we may take the limit $R\to\infty$ in the above inequality and end up with
\begin{equation}
	\lim_{s\to\infty} \big\|v(s)-V_*^{1/m}\big\|_{m+1}^{m+1} = 0. \label{fg04}
\end{equation}
Combining~\eqref{fg04} and the boundedness~\eqref{ubv} of $v$ in $L^\infty((\ln{m}/(m-1),\infty)\times\mathbb{R}^N)$ gives the convergence~\eqref{asympt.conv2} (with $\varphi:=V_*^{1/m}$) for all $r\in [m+1,\infty)$. 

\medskip

\noindent\textbf{Step~5.} We are left with the convergence in $L^\infty(\mathbb{R}^N)$. Let $\{s_j\}_{j\ge 1}$ be a sequence of positive real numbers such thart $s_j\to\infty$ as $j\to\infty$. Without loss of generality, we may assume that $s_j\ge 1+\ln{m}/(m-1)$ for $j\ge 1$. Setting $v_j(s) := v(s+s_j)$ for $s\in [-1,1]$ and $j\ge 1$, we argue as in the proof of Proposition~\ref{prop.omega} with the help of~\eqref{ubv} to show that
\begin{equation}\label{fg05}
	\lim_{j\to\infty} \sup_{s\in [-1,1]} \big\|v_j(s) - V_*^{1/m}\big\|_{2m} = 0,
\end{equation}
taking into account that $\omega(u_0)=\{V_*^{1/m}\}$ as shown in \textbf{Step~4}. Next, let $\varepsilon\in (0,1)$ and observe that, owing to~\eqref{eq2} and~\eqref{ubv}, $v_j$ is a weak solution to
\begin{equation*}
	\partial_s v_j - \Delta\Phi(v_j) = S_j, \quad (s,x)\in (-1,1)\times B(0,2/\varepsilon),
\end{equation*}
with 
\begin{equation*}
	\Phi(z) := \max\left\{ z , 2 C_2(u_0) \right\}^m, \quad z\ge 0, 
\end{equation*} 
and
\begin{equation*}
	S_j(s,x) := -|x|^\sigma v_j^m(s,x) + v_j(s,x), \quad (s,x)\in (-1,1)\times B(0,2/\varepsilon).
\end{equation*}
Since
\begin{equation*}
	|S_j(s,x)| \le \left( \frac{2}{\varepsilon} \right)^\sigma C_2(u_0)^m + C_2(u_0), \quad (s,x)\in (-1,1)\times B(0,2/\varepsilon),
\end{equation*}
the assumptions~($A_1$)-($A_6$) in \cite{PV93} are satisfied and we infer from \cite[Theorem~1.1]{PV93} that there are $\alpha_\varepsilon\in (0,1)$ and $\gamma_\varepsilon>1$ such that
\begin{equation*}
	|v_j(s,x) - v_j(s_0,x_0)| \le \gamma_\varepsilon \left( |x-x_0|^{\alpha_\varepsilon} + |s-s_0|^{\alpha_\varepsilon/2} \right) 
\end{equation*}
for $(s,s_0,x,x_0) \in [-1+\varepsilon,1-\varepsilon]^2\times \bar{B}(0,1/\varepsilon)^2$. Consequently,
\begin{equation*}
	\{v_j\}_{j\ge 1} \;\text{ is bounded in }\; C^{\alpha_\varepsilon/2,\alpha_\varepsilon}\big([-1+\varepsilon,1-\varepsilon]\times \bar{B}(0,1/\varepsilon)\big),
\end{equation*}
and we deduce from the Arzel\`a-Ascoli theorem and the convergence~\eqref{fg05} by a diagonal process that there is a subsequence $\{v_{j_k}\}_{k\ge 1}$ of $\{v_j\}_{j\ge 1}$ such that 
\begin{equation}
	v_{j_k} \longrightarrow V_*^{1/m} \;\;\text{ in }\;\; C\big([-1/2,1/2]\times \bar{B}(0,R)\big) \label{fg06}
\end{equation}
for any $R>1$. 

We now take advantage of the upper bound~\eqref{fg01}, along with the tail control~\eqref{tc}, to complete the proof. Indeed, pick $R>1$, $k\ge 1$ and $x\in\mathbb{R}^N$. Then, either $x\in B(0,R)$ and
\begin{align*}
	\big|v(s_{j_k},x) - V_*^{1/m}(x)\big| & = \big|v_{j_k}(0,x) - V_*^{1/m}(x)\big| \\
	& \le \left\| v_{j_k} - V_*^{1/m} \right\|_{C\big([-1/2,1/2]\times \bar{B}(0,R)\big)}.
\end{align*}
Or, $x\in\mathbb{R}^N\setminus B(0,R)$ and, according to~\eqref{fg01} with $s_0=s_{j_k}$ and Proposition~\ref{prop.exrad}, 
\begin{align*}
	\big|v(s_{j_k},x) - V_*^{1/m}(x)\big| & \le \big|v(s_{j_k},x)\big| + \big|V_*^{1/m}(x)\big| \\
	& \le V_*^{1/m}(x) + \left( \frac{e^{(m-1)s_{j_k}}}{e^{(m-1)s_{j_k} - 1}} \right)^{1/(m-1)} V_*^{1/m}(s,x) \\
	& \le C \left[ 1 + \left( \frac{e^{(m-1)s_{j_k}}}{e^{(m-1)s_{j_k} - 1}} \right)^{1/(m-1)} \right] R^{-\sigma/(m-1)}.
\end{align*}
Therefore,
\begin{align*}
	\big\| v(s_{j_k}) -V_*^{1/m}\big\|_\infty & \le  \left\| v_{j_k} - V_*^{1/m} \right\|_{C\big([-1/2,1/2]\times \bar{B}(0,R)\big)} \\
	& \quad + C \left[ 1 + \left( \frac{e^{(m-1)s_{j_k}}}{e^{(m-1)s_{j_k} - 1}} \right)^{1/(m-1)} \right] R^{-\sigma/(m-1)},
\end{align*}
and we may take the limit $k\to\infty$ in view of~\eqref{fg06} to obtain
\begin{equation*}
	\limsup_{k\to\infty} \big\| v(s_{j_k}) -V_*^{1/m}\big\|_\infty \le 2C R^{-\sigma/(m-1)}.
\end{equation*}
Letting $R\to\infty$ gives
\begin{equation*}
	\lim_{k\to\infty} \big\| v(s_{j_k}) -V_*^{1/m}\big\|_\infty = 0.
\end{equation*}
We have thus proved that, given an arbitrary sequence $\{s_j\}_{j\ge 1}$ satisfying $s_j\to\infty$ as $j\to\infty$, there is a subsequence $\{s_{j_k}\}_{k\ge 1}$ with $s_{j_k}\to\infty$ as $k\to\infty$ such that $\{v(s_{j_k})\}_{k\ge 1}$ converges in $L^\infty(\mathbb{R}^N)$ and its limit is $V_*^{1/m}$. Consequently, $\{v(s_j)\}_{j\ge 1}$ is relatively compact in $L^\infty(\mathbb{R}^N)$ with a unique cluster point $V_*^{1/m}$, whence $\{v(s_j)\}_{j\ge 1}$ converges to $V_*^{1/m}$ in $L^\infty(\mathbb{R}^N)$, and the proof of~\eqref{asympt.conv2} is complete, recalling that $\varphi=V_*^{1/m}$.
\end{proof}

\section{Non-existence of stationary solutions for $\sigma\leq\sigma_0$}\label{sec.nonex}

This section is devoted to showing the optimality of the range of $\sigma$ for the existence of a stationary solution to~\eqref{eq2} in $\mathcal{X}_+$. We will actually show that Eq.~\eqref{EE} has no solution in $\mathcal{Y}_+$ and, to this end, establish a Pohozaev identity following analogous steps as in \cite[Section~4]{IL25}.

\begin{proof}[Proof of Proposition~\ref{prop.ne}]
Assume that $\sigma\in [0,\sigma_0]$ and consider a solution
\begin{equation*}
	V\in \dot{H}^{1}(\mathbb{R}^N)\cap L^2_{\sigma}(\mathbb{R}^N)\cap L_{+}^{(m+1)/m}(\mathbb{R}^N)
\end{equation*}
to~\eqref{EE}. We then multiply~\eqref{EE} by $x\cdot\nabla V$ and integrate by parts. By proceeding exactly as in \cite[Section~4]{IL25}, we derive the following identities
\begin{equation*}
\begin{split}
	&-\int_{\mathbb{R}^N}(x\cdot\nabla V)\Delta V\,dx==-\frac{N-2}{2}\|\nabla V\|_2^2,\\
	&\int_{\mathbb{R}^N}(x\cdot\nabla V)|x|^{\sigma}V\,dx=-\frac{N+\sigma}{2}\mathcal{N}_{\sigma}^2(V),\\
	&\int_{\mathbb{R}^N}(x\cdot\nabla V)V^{1/m}\,dx=-\frac{mN}{m+1}\int_{\mathbb{R}^N}V^{(m+1)/m}\,dx.
\end{split}
\end{equation*}
Gathering the previous identities, we infer from Eq.~\eqref{EE} that
\begin{equation}\label{poh1}
	-\frac{N-2}{2}\|\nabla V\|_{2}^2 - \frac{N+\sigma}{2} N_{\sigma}^2(V)=-\frac{mN}{m+1}\|V\|_{(m+1)/m}^{(m+1)/m}.
\end{equation}
Multiplying Eq.~\eqref{EE} by $V$ and integrating over $\mathbb{R}^N$, we find
\begin{equation}\label{poh2}
	\|\nabla V\|_{2}^2+N_{\sigma}^2(V)=\|V\|_{(m+1)/m}^{(m+1)/m}.
\end{equation}
We next multiply~\eqref{poh2} by $(N+\sigma)/2$ and add the resulting equality to~\eqref{poh1} in order to eliminate the terms containing $N_{\sigma}^2(V)$. We thus obtain
\begin{equation}\label{poh3}
	\frac{\sigma+2}{2}\|\nabla V\|_2^{2}+\frac{N(m-1)-\sigma(m+1)}{2(m+1)}\|V\|_{(m+1)/m}^{(m+1)/m}=0.
\end{equation}
Since $\sigma\in[0,\sigma_0]$, we observe that the coefficient of the first term in~\eqref{poh3} is strictly positive, while the coefficient of the second term in~\eqref{poh3} is non-negative. This readily implies that $\|\nabla V\|_{2}=0$, hence $V$ is constant for a.a. $x\in\mathbb{R}^N$. We then deduce from the integrability properties of $V$ that $V\equiv0$. Letting $\varphi:=V^{1/m}$ completes the proof.
\end{proof}

\section*{Acknowledgements}
This work is partially supported by the Spanish project PID2024-160967NB-I00. Part of this work was done while PhL enjoyed the hospitality of Universidad Rey Juan Carlos.

For the purpose of Open Access, a CC-BY license has been applied by the authors to all versions of this article leading up to the Author Accepted Manuscript.


\section*{Data availability}

Our manuscript has no associated data.

\section*{Conflict of interest}

The authors declare that there is no conflict of interest.

\bibliographystyle{plain}


\end{document}